\documentclass[12pt]{article}
\usepackage{amssymb,epsfig,amsfonts,amsmath,amscd,latexsym,amsfonts}

\usepackage[small,nohug,heads=vee]{diagrams}
\diagramstyle[labelstyle=\scriptstyle]

\usepackage{color}

\newtheorem{defin}{Definition}
\newtheorem{prop}{Proposition}
\newtheorem{nt}{Remark}
\newtheorem{Th}{Theorem}

\newfont{\ssdbl}{msbm8}
\newfont{\sdbl}{msbm9}
\newfont{\dbl}{msbm10 at 12pt}

\newcommand{\Ob}{\mathop {\rm Ob}}

\newcommand{\oo}{{\cal O}}

\newcommand{\ad}{{\cal A}}

\newcommand{\diag}{\mathop {\rm diag}}

\newcommand{\Ind}{\mathop {\rm Ind}}

\newcommand{\Hom}{\mathop {\rm Hom}}
\newcommand{\Aut}{\mathop {\rm Aut}}

\newcommand{\Ext}{\mathop {\rm Ext}}
\newcommand{\Lim}{\mathop {\rm lim}}
\newcommand{\LLim}{\mathop {\rm 2lim}}

\newcommand{\Spec}{\mathop {\rm Spec}}

\newcommand{\Frac}{\mathop {\rm Frac}}

\newcommand{\Id}{\mathop {\rm Id}}

\newcommand{\Vect}{\mathop {\rm Vect}}

\newcommand{\da}{\mathbb{A}}
\newcommand{\dz}{\mathbb{Z}}
\newcommand{\dc}{\mathbb{C}}
\newcommand{\dqq}{\mathbb{Q}}
\newcommand{\dr}{\mathbb{R}}

\newcommand{\Ker}{{\rm Ker}\:}

\newcommand{\lto}{\longrightarrow}

\newcommand{\df}{\mathbb{F}}

\newcommand{\Gal}{\mathop {\rm Gal}}
\newcommand{\BB}{{\cal B}}
\newcommand{\CC}{{\cal C}}
\newcommand{\ve}{\varepsilon}

\newcommand{\VV}{{\mathcal{V}}}
\newcommand{\CCC}{{\mathbb{C}}}

\begin{document}

\author{D.V. Osipov\footnote{This work was partially supported by
RFBR (grants no.~11-01-00145-a, no.~11-01-12098-ofi-m-2011, and no.~12-01-33024 mol\_a\_ved), by
President's Programme for Support of Leading Scientific Schools
(grant no.~NSh-5139.2012.1).}}

\title{Unramified two-dimensional  Langlands correspondence}
\date{}

\maketitle

\qquad \qquad \qquad \qquad \qquad \qquad \qquad \qquad \qquad \qquad
{\em Dedicated to I.R. Shafarevich on}

\qquad \qquad \qquad \qquad \qquad \qquad \qquad \qquad \qquad \quad
{\em  the occasion of his ninetieth birthday}


\abstract{In this paper we describe the unramified Langlands correspondence for two-dimensional local fields,
we construct a categorical analogue of the unramified principal series representations and study its properties.
The main tool for this description is the construction of a central extension.
For this (and other) central extension we prove noncommutative
reciprocity laws (i.e. the splitting of the central extensions over some subgroups)
for arithmetic surfaces and projective surfaces over a finite field.
These reciprocity laws connect central extensions which are constructed locally and globally.}

\section{Introduction}
In 1995, M.~Kapranov proposed in~\cite{K}  a very hypothetical
generalization of the Langlands program for the two-dimensional case.
The classical Langlands correspondence deals with a one-dimensional
base: with number fields or with function fields of algebraic curves
defined over a finite field. In~\cite{K}  the question was raised on  the
existence of an analogue of the Langlands program for arithmetic
surfaces or surfaces defined over a finite field. The paper~\cite{K} did no contain constructions,  but the case of the
abelian two-dimensional (local) Langlands correspondence   was treated there.

The main idea proposed in~\cite{K} is that to an $ n $-dimensional representation of the group $ \Gal (K^{\rm sep} / K) $,
where 1) $ K $ is a two-dimensional local field\footnote{In this paper, by a two-dimensional local field $ K $ we mean a complete discrete valuation field with residue field $ \bar{K} $ such that  the field
$ \Bar{K} $ is a one-dimensional local field, that is a complete discrete valuation field with the finite residue field $ \df_q $. (A basic survey of the first notions of
two-dimensional local fields and two-dimensional adeles is contained in~\cite{O1}.)} or 2) $ K $ is  the  field of rational functions of an arithmetic surface or of a projective algebraic surface over a finite field,  one must somehow attach a (categorical) representation of the group
1) $ GL_{2n} (K) $ or 2) $ GL_{2n} (\da) $ in a $ 2 $-vector space. (Here $ \da $ is the ring of two-dimensional Parshin-Beilinson adeles,
see~\cite{H}, of a given  arithmetic surface or of a projective algebraic surface over a finite field.  In the case of an arithmetic surface this ring must take into account the fibers over infinite (archimedean) points of the base, see~\cite[Example~11]{OP}, \cite{P1}.)

In the case $ n = 1 $, the one-dimensional (classical) Langlands correspondence  is the statement of the one-dimensional (ordinary) class field theory. The two-dimensional class field theory was developed by A.N.~Parshin, K.~Kato and others, see survey~\cite{Ra}. In the local case it is reduced to the description of the abelian Galois group of a two-dimensional local field. This description is based on the construction of the reciprocity map:
 \begin{equation} \label{vz}
 K_2 (K) \lto \Gal (K^{\rm {ab}} / K) \mbox {,}
 \end{equation}
where $ K^{\rm {ab}} $ is the maximal abelian extension of a two-dimensional local field $ K $.

We note that recently A.N.~Parshin formulated in~\cite{P1} a
hypothesis on the direct image of automorphic forms, which connects
the two-dimensional abelian Langlands correspondence with the
classical Langlands correspondence in dimension one. The classical
Hasse-Weil conjecture on the existence of a functional equation for
$ L $-functions of arithmetic surfaces follows from this hypothesis.

In this paper we describe the unramified Langlands correspondence for two-dimensional local fields, we construct a categorical analogue of the unramified principal series representations of the group $ GL_{2n}(K) $, where $ K $ is a two-dimensional local field and the group $ GL_{2n} ( K) $ acts on the abelian $ \dc $-linear category, and we study the properties of the constructed categorical representations. The main tool for this goal
is a construction of some central extension of the group $ GL_2 (K) $ by the group $ \dr_+^* $. For this (and other) central extension we prove the non-commutative  reciprocity laws (i.e. splitting of central extensions over some subgroups) for arithmetic surfaces and projective surfaces over finite fields. These reciprocity laws connect  central extensions which are constructed locally and globally.

The paper is organized as follows.

In \S~\ref{abcase}  we recall the abelian case of the two-dimensional local Langlands correspondence.  In \S~\ref{odn} we connect  the one-dimensional $ 2 $-representations of a group  and the central extensions of the same group. In \S~\ref{Steinberg} we calculate the second cohomology group $ GL_n (K) $ for an infinite field $ K $.

In \S~\ref{C2} we consider central extensions which are constructed by $ C_2 $-spaces. In \S~\ref{C2fin} we recall the definition of the category
$ C_2^{\rm fin} $.
As examples of objects in this category we have  two-dimensional local fields, adelic rings of arithmetic surfaces or of algebraic surfaces over finite fields.
In \S~\ref{constr} we construct some central extensions of the group $ GL_n (\da_{\Delta}) $, where $ \da_{\Delta} $ is a  subring of the adelic ring of an arithmetic surface or of an algebraic surface over a finite field.
 In \S~\ref{dv} the constructed central extensions are studied in the case when $ \da_{\Delta} $ is a finite product of two-dimensional local fields. In \S~\ref{ar} we construct and study certain central extensions of the groups $ GL_n (\dr ((t))) $ and $ GL_n (\dc ((t))) $. These central extensions are used for the construction of central extensions of the group  $ GL_n (\da_X^ {\rm ar}) $, where $ \da_X^ {\rm ar} $ is the arithmetical  adelic ring of an arithmetic surface $ X $ (i.e., the adelic ring which  takes into account the archimedean fibers).
In \S~\ref{noncom} we prove non-commutative  reciprocity laws: in theorem~\ref{t1}
we prove that the central extensions which are constructed from the whole adelic  ring  $ \da $  of a projective surface over a finite field or of an arithmetic surface  split over some subgroups of $ GL_n (\da) $. These subgroups are important to describe the semilocal situation on the  scheme under consideration and are related either to  closed points, or to  integral one-dimensional subschemes of this scheme.

In \S~\ref{tlf} we consider the unramified Langlands correspondence for two-dimensional local fields. In \S~\ref{one-dim} we recall the  classical construction   of this correspondence for a one-dimensional local field. In \S~\ref{gract} we recall the known facts, which we need further, on the action of a group on a $ k $-linear category, where $ k $ is a field. In \S~\ref{catan} we construct a categorical analogue of the unramified principal series  representations of the group $ GL_{2n} (K) $, where $ K $ is a two-dimensional local field. We define also here the notion of a smooth action of the group $ GL_l (K) $ on a $ k $-linear abelian category (we call this category as a generalized $ 2 $-vector space) so that the constructed categorical analogue of the principal series representations will be smooth (in this case $ k = \dc $). This (and other) property of these representations is proved in theorem~\ref{th2}.
We note that to construct the categorical principal series representations for the group $ GL_{2n} (K) $ we use an analogue of the induced representation (for categories), as well as a central extension of the group $ GL_2 (K) $, which is  constructed in \S~\ref{C2} and is connected with  the unramified  class field theory for the field $ K $. In \S~\ref{hyp} we discuss some hypothesis about the relationship of a smooth spherical action of the group $ GL_{2n} (K) $ on any $ \dc $-linear abelian category with a categorical principal series representation  of the group $ GL_{2n} (K ) $ (this representation was constructed in \S~\ref{catan}).

This article was largely written during  the visit of the
author to the  Max Planck Institute for Mathematics  in Bonn in March
2012. The author is grateful to this Institute
for the excellent working conditions. I am also grateful to
S.O.~Gorchinskii and to A.N~Parshin for their valuable comments after my
talk at the seminar on the arithmetic algebraic geometry in the Steklov Mathematical Institute.

\section{The abelian case of two-dimensional Langlands correspondence} \label{abcase}
\subsection{Central extensions and one-dimensional $ 2 $-vector spaces} \label{odn}
We recall that, by definition,  a finite-dimensional $ 2 $-vector space $ \cal C $ over a field $ k $  is a  semisimple abelian $ k $-linear category such that  a  set of
isomorphism classes of simple objects is finite,  and  for any simple object $ E \in \cal C $ holds
 $ \Hom_ {\cal C} (E, E) \simeq k $, see~\cite{KV}. (Under a $ k $-linear category we mean a category $ \cal C $, in which for any two objects $ A $ and $ B $ the set $ \Hom_{\cal C} (A, B) $ is a  $ k $-vector space, and the composition of morphisms is bilinear. A category is called semisimple if every object in it is a finite direct sum of simple objects.) Number of elements in the set
of isomorphism classes of simple objects is called the dimension of a $ 2 $-vector space. It is easy to see that $ n $-dimensional $ 2 $-vector space
is equivalent to the category  $ ({\Vect_k^ {\rm fin}})^{n} $, where $ \Vect_k^ {\rm fin} $ is a category of finite-dimensional $ k $-vector spaces.

We
consider a one-dimensional $ 2 $-vector space $ \cal C $.
 In the one-dimensional $ 2 $-vector space $ \cal C $  all simple objects are isomorphic, and any $ k $-linear endofunctor $ F: \cal C \to \cal C $ up to isomorphism is given by a finite-dimensional vector space over $ k $:
$$
V = \Hom \nolimits_{\cal C} (E, F (E)) \mbox {,}
$$
where $ E $ is a simple object in $ \cal C $.
Conversely, a one-dimensional $ 2 $-vector space is equivalent to the category $ \Vect_k^{\rm fin} $. Therefore any finite-dimensional vector space
 $ V \in \Vect_k^{\rm fin} $ defines a $ k $-linear endofunctor on the category $ \Vect_k^{\rm fin} $ by the rule:
 $$
X \longmapsto V \otimes_k X \mbox {,} \quad \mbox {where} \quad X \in \Vect \nolimits_k \nolimits^{\rm fin} \mbox {.} $$

Let $ G $ be a group. By definition, a representation of the group $ G $ in a $ 2 $-vector space $ \cal C $ (of any dimension) is a strongly monoidal functor from the discrete category of $ G $ (with objects as elements of $ G $ and only the identity morphisms of objects) to the monoidal category of functors which are $ k $-linear equivalences of the category $ \cal C $, see~\cite{GK}, \cite[\S~5.3]{K}, and also \S~\ref{gract} later on. Therefore any $ 2 $-representation of the group  $ G $ in a one-dimensional $ 2 $-vector space  is defined up to equivalence by  $ 1 $-dimensional $ k $-vector spaces $ V_g $ (for every element $ g $ of the group $ G $) together with natural isomorphisms
$$ V_{g_1} \otimes_k V_{g_2} \longrightarrow V_{g_1g_2}  $$
which satisfy the associativity condition for any elements $ g_1, g_2, g_3 $ of the group $ G $. These data correspond to a central extension of
the group $ G $:
$$
1 \lto k ^ * \lto \widehat{G} \stackrel{\pi}{\lto} G \lto 1 \mbox {,}
$$
where $ k^* $-torsor $ \pi^{-1} (g) $ is equal to $ V_g \setminus 0 $. Hence we conclude that the set of equivalence classes of one-dimensional
$ 2 $-representations of the group $ G $ coincides with the set $ H^2 (G, k^*) $.
\begin{nt} {\em
The category of representations of any group (even more, of an arbitrary $ 2 $-group) in a finite-dimensional $ 2 $-vector space has been in detail studied in~\cite{Elg}.}
\end{nt}

\subsection{The second cohomology groups of the group $ GL_n (K) $} \label{Steinberg}
Let $ A $ be any abelian group, $ K $ be an arbitrary infinite field. Following~\cite[\S~1]{D}
(see also~\cite{Bu}), we construct an explicit map
\begin{equation} \label{expl}
\Hom (K_2 (K), A) \lto H^2 (GL_2 (K), A) \mbox {.}
\end{equation}
Let us consider the universal central extension
\begin{equation} \label{uext}
1 \lto K_2 (K) \lto St (K) \lto SL (K) \lto 1 \mbox {,}
\end{equation}
where $ St (K) $ is the Steinberg group of the field $ K $, and $ SL (K) = \mathop {\underrightarrow {\lim}} \limits_n SL_n (K) $. Let us consider the central extension which is a pullback
of the central extension~\eqref{uext} under the embedding of the group $ SL_2 (K) $ to the group $ SL (K) $:
$$
1 \lto K_2 (K) \lto \widetilde{SL_2 (K)} \lto SL_2 (K) \lto 1 \mbox {.}
$$
We  have
 also 
 \begin{equation} \label {coinv}
 K_2 (K) = H_2 (SL (K), \mathbb{Z}) = H_2 (SL_2 (K), \mathbb{Z})_{K^*} \mbox {.}
 \end{equation}
The  group $ GL_2 (K) = SL_2 (K) \rtimes K^* $,
where the group $ K^* = GL (1, K) \hookrightarrow GL_2 (K) $ (embedding in the upper-left corner) acts on
the group $ SL_2 (K) $ by conjugations, which are inner automorphisms of the group $ GL_2 (K) $.
The group $ K^* $ acts also on the group $ SL (K) $ by conjugations.
Because of the universality of the central extension~\eqref{uext} we obtain the lift of the action of the group $ K^* $
on $ St (K) $, whose restriction to the kernel $ K_2 (K) $ is trivial by formula~\eqref{coinv}.
Hence we have a central extension
\begin{equation} \label{cenext}
1 \lto K_2 (K) \lto \widehat{GL_2 (K)} \stackrel {\pi} {\lto} GL_2 (K) \lto 1
\end{equation}
which splits over the subgroup $ K^* $, and $ \widehat {GL_2 (K)} = \widetilde {SL_2 (K)} \rtimes K^* $. The universal symbol $ (x, y) \in K_2 (K) $
is also obtained in the following way (see~\cite[\S~1.9]{D}):
\begin{equation}  \label{equ}
(x,y) = \left< \left(
\begin{matrix}
y & 0\\
0 & 1
\end{matrix}
\right),
\left( \begin{matrix}
1 & 0 \\
0 & x
\end{matrix}
\right) \right>
  \mbox{,}
\end{equation}
where for all commuting matrices $ A $ and $ B $ from $ GL (2, K) $
the element \linebreak $ <A,B> = [\hat{A}, \hat{B}] $ from $ K_2 (K) $
is well-defined, and
$ \pi(\hat{A}) = A $,
$ \pi (\hat{B}) = B $. The element $ <A,B>$   does not depend on the choice of suitable elements $ \hat{A} $ and $ \hat{B} $ from the group
$ \widehat{GL_2 (K)} $. Applying now an
element from $ \Hom (K_2 (K), A) $ to the kernel of  central extension~\eqref{cenext}, we obtain an explicit
description of map~\eqref{expl}.

On the other hand, from the Hochschild-Serre spectral sequence we have an exact sequence:
$$
0 \lto H^2 (K^*, A) \lto H^2 (GL_2 (K), A) \stackrel {\alpha} {\lto} H^2 (SL_2 (K), A)^{K^*} \mbox {.}
$$
(It's also worth to note that $ H^1 (SL_2 (K), A) = \Hom (SL_2 (K), A) = 0 $, since the group $ SL_2 (K) $ is a perfect group.)

By the universal coefficient formula and since the group $ SL_2 (K) $ is a perfect group, we obtain that
 $$ H ^2(SL_2 (K), A) = \Hom (H_2 (SL_2 (K), \mathbb{Z}), A) \mbox {.} $$
 Consequently, we have
$$
H^2 (SL_2 (K), A)^{K^*} = \Hom (H_2 (SL_2 (K), \mathbb{\dz})_{K^*}, A) = \Hom (K_2 ( K), A) \mbox {.}
$$
Besides, map~\eqref{expl} is a section of map~$\alpha $. Summing up the above, we obtain  the following statement.
\begin{prop} \label{prop-coh}
Let $ A $ be an abelian group, $ K $ be an infinite field. Then we have
$$
H^2 (GL_2 (K), A) = H^2 (K^*, A) \oplus \Hom (K_2 (K), A) \mbox {.}
$$
\end{prop}
\begin{nt} \label{dec-coh} {\em
Similarly, we can prove a more general statement for any $ n \ge 2 $ and an infinite field $ K $:
\begin{equation} \label{gen}
H^2 (GL_n (K), A) = H^2 (K^*, A) \oplus \Hom (K_2 (K), A) \mbox {.}
\end{equation}
}
\end{nt}

Combining the result of proposition~\ref{prop-coh} with  reciprocity map~\eqref{vz}, we obtain a map from the characters of the Galois group of a two-dimensional local field
$ K $ to the one-dimensional $2$-representations of the group $ GL (2, K) $. Hence, taking into account the topology on the group $ K_2 (K) $, we obtain a complete description of the two-dimensional abelian Langlands correspondence.

\section{$ C_2 $-spaces and central extensions} \label{C2}
\subsection{The category $ C_2^{\rm fin} $} \label{C2fin}
In the papers~\cite{Osip} and~\cite[\S~3.1]{OP} the  categories $ C_n^{\rm fin} $ for $ n \ge 0 $ were constructed. We will be interested in the category $ C_2^{\rm fin} $.
This category is constructed by induction.

The category $ C_0^{\rm fin} $ is the category of finite abelian groups.

The objects  of the category $ C_1^{\rm fin} $    are
filtered abelian groups $ (I, F, V) $. Here $ V $ is an abelian group, $ I $ is a partially ordered set such that
 for any elements $ i, j \in I $ there exist elements $ k, l \in I $ with the property $ k \le i \le l $, $ k \le j \le l $. And also $ F $  is a function from the set $ I $ to the set of subgroups of $ V $ such that if $ i \le j $ then
$ F (i) \subset F (j) $. In addition, we demand that $ \bigcap_{i \in I} F (i) = 0 $ and $ \bigcup_{i \in I} F (i) = V $.
The inductive step is that the group $ F (j) / F (i) $ is a finite group for any $ j \ge i \in I $, i.e.  this group is an object of the category
$ C_0^ {\rm fin} $. Morphisms between  constructed objects copy  the definition of continuous morphisms in one-dimensional local fields. More precisely, if $ E_1 = (I_1, F_1, V_1) $ and  $ E_2 = (I_2, F_2, V_2) $ are
objects from the category $ C_1^{\rm fin} $, then $ \Hom_{C_1^{\rm fin}} (E_1, E_2) $ consists of all homomorphisms of abelian groups $ f: V_1 \to V_2 $ such that the following  condition holds: for any $ j \in I_2 $ there exists $ i \in I_1 $ such that $ f (F_1 (i)) \subset F_2 (j) $.
 Examples of objects from the category  $ C_1^{\rm fin} $ are one-dimensional local fields,   the adelic rings of curves over finite fields, and  the rings of finite adeles of number fields.

The objects of the category $ C_2^{\rm fin} $ are filtered abelian groups $ (I, F, V) $ with the following condition: for any elements
$ i \le j \in I $, the structure of an object from the category of $ C_1^{\rm fin} $ is given on the group
  $ F (j) / F (i) $. This requires some coordination of structures of objects from
$ C_1^{\rm fin} $ for all $ i \le j \le k \in I $, see the exact definition in~\cite[Def.~4]{OP}. Morphisms between objects in the category of
$ C_2^{\rm fin} $ are defined inductively by means of  morphisms from the category $ C_1^{\rm fin} $, which are  defined on the quotient groups of filtrations, see the precise definition in~\cite[Def.~5]{OP}. Examples of objects from the category $ C_2^{\rm fin} $ are two-dimensional local fields,   adelic rings of algebraic surfaces over finite fields and adelic rings of arithmetic surfaces (excluding the fibers over the archimedean places), see~\cite[Th.~2.1]{Osip},~\cite[Example 1]{OP}.

\subsection {A central extension  $ \widehat{GL_n (\da_{\Delta})}_{\dr_ +^*} $} \label{constr}
Let $ K $ be a one-dimensional local field with the finite residue field $ \df_q $. Then there is a homomorphism
\begin{equation} \label{1det}
 GL (n, K) \lto \dr_+^* \quad : \quad A \mapsto q^{\nu (\det (A))} \mbox{,}
\end{equation}
where $ \nu $ is the discrete
valuation of the field $ K $. We see that this homomorphism comes from the homomorphism $ K^* \lto \dz $, i.e  from  the element of
$ H^1 (K^*, \dz) $.
In addition, there is a homomorphism: $ GL (n, \da^ {\rm fin}) \lto \dr_+^* $, where $ \da^{\rm fin} $ is the ring of finite adeles  of a curve over a finite field or of a number field. This homomorphism  is obtained by multiplying the local maps \eqref{1det}.

Now let $ X $ be an integral two-dimensional normal scheme of finite type over $ \dz $ (for example, a surface over a finite field or an arithmetic  surface).
Let $ \Delta $ be a subset in the set of all pairs $ x \in C $, where $ x \in X $ is a closed point, and $ C $ is an integral   one-dimensional subscheme
of $ X $ such that this subscheme  contains the point $ x $.
Inside of the ring $ \prod_{x \in C} K_{x, C} $ we define  the following subrings:
$$
\da_{\Delta} = \da_X \cap \prod_{\{ x \in C \} \in \Delta} K_{x, C} \mbox{,} \qquad \quad
\oo_{ \da_{\Delta}} = \da_X \cap \prod_{\{ x \in C \} \in \Delta } \oo_{K_{x, C}} \mbox{.}
$$
where $ \da_X $ is the adelic ring of the scheme $ X $, the  ring\footnote{See also further the beginning of section~\ref{dv}, where this ring is  described in detail.} $ K_{x, C} = \prod_{i } K_i $ is the finite product of two-dimensional local fields $ K_i $, constructed by a pair $ x \in C $,
$ \oo_{K_{x, C}} = \prod_i \oo_{K_i} $, where $ \oo_{K_i} $ is
the rank $ 1 $ discrete  valuation ring  of  the field $ K_i $.
Note that if $ \Delta = \Delta_1 \cup \Delta_2 $ and $ \Delta_1 \cap \Delta_2 = \emptyset $, then
$$
\da_{\Delta} = \da_{\Delta_1} \times \da_{\Delta_2} \mbox{,} \qquad \oo_{\da_{\Delta}} = \oo_{\da_{\Delta_1}} \times
\oo_{\da_{\Delta_2}} \mbox{,}
$$
\begin{equation} \label{decom}
GL_n (\da_{\Delta}) = GL_n (\da_{\Delta_1}) \times GL_n (\da_{\Delta_2}) \mbox {.}
\end{equation}

  We will construct a central extension for any $ n \ge 1 $:
\begin{equation} \label{ext1}
1 \lto \dr_+^* \lto \widetilde{GL_n (\da_ {\Delta})}_{\dr_+^*} \stackrel{\theta}{\lto} GL_n (\da_{\Delta}) \lto 1 \mbox {.}
\end{equation}

Let us consider  the following $ C_2^{\rm fin} $-structure $ (I, F, \da_{\Delta}^n) $ on $ \da_{\Delta}^{n}$, where the set
$$
I = \left \{ \oo_{\da_{\Delta}} - \mbox{submodules} \quad T \subset \da_{\Delta}^n \; \mid \; T = g \oo_{\da_{ \Delta}}^{\, n} \; \,
\mbox{for some} \quad g \in GL_n (\da_{\Delta})
 \right \}
$$
is ordered by embeddings of  submodules, and the function $ F $ maps an element from $ I $ to the corresponding submodule of $ \da_{\Delta}^n $.
For any elements $ i \le j \in I $, we have that
$ \oo_{\da_{\Delta}} $-module $ F (j) / F (i) $ is a locally compact abelian group, where the topology on this group is considered as the quotient and the  induced topology  from the topological group $ \da_{\Delta}^n $. (The group $ \da_{\Delta} $ is endowed with the natural topology of iterated inductive and projective limits by virtue of its construction.) Let us  consider a filtration on $ F (j) / F (i) $ given by  open compact subgroups. This filtration sets on $ F (j) / F (i) $ the structure of  an object  from the category of $ C_1^{\rm fin} $. These structures will be coherent in the sense of $ C_2^{\rm fin} $-structure for different $ i \le j \le k \in I $.

For any locally compact abelian group $ E $ let $ \mu(E) $ be the canonical $ \dr_+^* $-torsor  of nonzero Haar measures on $ E $. For any elements
$ i, j \in I $ we consider an \linebreak
$ \dr_+^* $-torsor $ \mu (F (i) \mid F (j)) $ which is canonically defined by the following conditions:
\begin{enumerate}
\item $ \mu (F (i) \mid F (j)) \otimes \mu (F (j) \mid F (k)) = \mu (F (i) \mid F (k)) $ for any $ i, j, k \in I $;
\item $ \mu (F (i) \mid F (j)) = \mu (F (j) / F (i)) $ for all $ i \le j \in I $ \mbox {.}
\end{enumerate}
Any element $ g \in GL_n (\da_{\Delta}) $ moves the $ \dr_+^* $-torsor $ \mu (F (i) \mid F (j)) $ to the $\dr_+^* $-torsor $ \mu (gF (i) \mid gF (j)) $.
Let us {\em define} a group
$$ \widetilde {GL_n (\da_{\Delta})}_{\dr_+^*} = \left \{(g, \mu) \; \mid \; g \in GL_n (\da_{\Delta }), \;
\mu \in \mu (\oo_{\da_{\Delta}}^{\, n} \mid g \oo_{\da_{\Delta}}^{\, n}) \right \} \mbox { .} $$
 The multiplication law in this group is defined as
$$ (g_1, \mu_1) (g_2, \mu_2) = (g_1g_2, \mu_1 \otimes g_1 (\mu_2)) $$
with the obvious identity element and the construction of the
inverse element. The map $ \theta: (g, \mu) \mapsto g $ defines
central extension~\eqref{ext1}.

The group $ GL_n (\da_ {\Delta}) = SL_n (\da_{\Delta}) \rtimes
\da_{\Delta}^* $ (the group $ \da_{\Delta}^* $ is embedded into the
upper left corner of the group $ GL_n (\da_{\Delta}) $, see
also~\S\ref{Steinberg}). With the help of the central
extension~\eqref{ext1} the action of the group $\da_{\Delta}^* $ is
uniquely lifted to an action on the group $ \theta^{-1} (SL_n
(\da_{\Delta})) $ (by means of the inner automorphisms in the group
$ \widetilde {GL_n (\da_{\Delta})}_{\dr_+^*}) $). This action
becomes trivial after its restriction  to the kernel $ \dr_+^* $.
Let us {\em define} a group $ \widehat {GL_n
(\da_{\Delta})}_{\dr_+^*} = \theta^{-1}(SL_n (\da_ {\Delta}))
\rtimes \da_{\Delta}^* $. We obtain a central extension
\begin{equation} \label {ext2} 1 \lto \dr_+^* \lto
\widehat{GL_n (\da_{\Delta})}_{\dr_+^*} {\lto} GL_n (\da_{\Delta})
\lto 1 \mbox {,}
\end{equation}
which splits over the  subgroup  $ \da_{\Delta}^* $ of the group
$GL_n (\da_{\Delta}) $.

\begin{nt} {\em
The central extension  $ \widehat{GL_n
(\da_{\Delta})}_{\dr_+^*} $ is  not isomorphic to the central
extension  $ \widetilde{GL_n (\da_{\Delta})}_{\dr_+^*} $, since the
first central extension splits over the subgroup  $ \da_{\Delta}^* $
of the group $ GL_n (\da_{\Delta}) $, while the second central
extension does not split over the same subgroup. The latter follows,
for example, from the explicit calculation of the commutator of the
lifting   of some elements from the commutative subgroup   $
\da_{\Delta}^* $.  This commutator is not equal to zero on some
elements, but it would be identically equal to zero on all  elements
in the case of a split central extension.

 To define the central extension  $ \widehat{GL_n
(\da_{\Delta})}_{\dr_+^*} $, we have embedded the group $
\da_{\Delta}^* $ into the upper left corner of the group $ GL_n
(\da_ {\Delta}) $. If we embed this group at any $ i $-th (where $ 2
\le i \le n $) place on the diagonal, then the newly constructed
central extension of the group  $ GL_n (\da_{\Delta}) $ is
canonically isomorphic to the central extension $ \widehat {GL_n
(\da_{\Delta})}_{\dr_+^*} $ of the same group. We show it by
assuming for simplicity $ n = 2 $. We fix a matrix $ A =
\left(\begin{matrix} 0 & 1 \\ -1 & 0 \end{matrix} \right) $, which
conjugates the matrices $ \diag (b, 1) $ and $ \diag (1, b) $ in the
group $ GL_2 (\da_{\Delta}) $. Let $ \phi_A $ be  the inner
automorphism of the group $ \theta^{-1}(SL_2 (\da_{\Delta})) $ which
is defined by the lifting  $ \hat{A} $ of the element  $ A $ to this
group (the automorphism $ \phi_A $ does not depend on the choice of the
lifting). The homomorphisms $ \da_{\Delta}^* \to \Aut (\theta^{-1}
(SL_2 (\da_{\Delta}))) $ which are constructed from the embedding of
the group $ \da_{\Delta}^* $ into different places on the diagonal
of the group $ GL_2 (\da_{\Delta})$ distinct by  the conjugation by
the element  $ \phi_A $ in the group $ \Aut (\theta^{-1}
(SL_2(\da_{\Delta}))) $. Therefore the map $ \{x, b \} \mapsto
\{\phi_A (x), b \} $, where $ \{x, b \} \in \theta^{-1} (SL_2
(\da_{\Delta})) \rtimes \da_{\Delta}^* $, is an isomorphism of the
two semidirect products which are constructed by embeddings of the
group $ \da_{\Delta}^* $ into different places on the diagonal. This
isomorphism induces an inner automorphism of the group $ GL_2
(\da_{\Delta}) $ such that it is defined by the element $ A $. It
remains to note that any inner automorphism of the group induces the
canonical automorphism of a central extension of the same group.  }
\end{nt}

We note also that from the construction it follows  at once that the
central extensions  $ \widetilde{GL_n (\da_{\Delta})}_{\dr_+^*} $
 and $ \widehat{GL_n (\da_{\Delta})}_{\dr_+^*} $ canonically split
over the subgroup $ GL_n (\oo_{\da_{\Delta}}) $ of the group $ GL_n
(\da_{ \Delta}) $.

\begin{prop} \label{prope} The  central extensions so constructed satisfy
the following properties.
\begin{enumerate}
\item \label {it1}
If
$ \Delta = \Delta_1 \cup \Delta_2 $ and $ \Delta_1 \cap \Delta_2 = \emptyset $, then the central extension
$ \widehat{GL_n (\da_{\Delta})}_{\dr_+^*} $ is the Baer sum of the central
extensions $ \widehat{GL_n (\da_ {\Delta_1})}_{\dr_+^*} $ and $ \widehat{GL_n (\da_{\Delta_2}) }_{\dr_+^*}
$ (with respect to the projections onto the direct summands in  expansion~\eqref{decom}).
\item \label{it2} A central extension $ \widehat {GL (\da_{\Delta})}_{\dr_+^*}$  of the group
$ GL (\da_{\Delta}) = \mathop{\underrightarrow{\lim}} \limits_n GL_n (\da_{\Delta}) $
is well-defined by the central extensions
$ \widehat{GL_n (\da_{\Delta})}_{\dr_+^*} $ ($ n \ge 1 $).
\end{enumerate}
\end{prop}
\begin{proof} Clearly, it is sufficient to
prove similar results for the central extensions (and groups) $ \widetilde{GL_n (\da_{\Delta})}_{\dr_+^*} $.
Item~\ref{it1} follows from the construction of  central extension~\eqref{ext1} and the following properties.
If
$$ 0 \lto V_1 \lto V \lto V_2 \lto 0 $$ is an exact sequence of locally compact abelian
groups, where all morphisms are continuous, and
$ V_1 \hookrightarrow V $ is a closed embedding, then there is a canonical
isomorphism $ \mu_{V_1, V_2} $ of the following $ \dr_+^* $-torsors:
$$ \mu(V_1) \otimes \mu(V_2) \lto \mu (V) \mbox{.} $$
Here, if $ V = V_1 \oplus V_2 $, then for any elements $ v_1 \in \mu (V_1) $, $ v_2 \in \mu (V_2) $
there is an equality\footnote{We note
that if we assume that $ V_i $ $ (1 \le i \le 2) $ are finite-dimensional
vector spaces over a field and replace $ \mu(V_i) $ by $ \det V_i = \wedge^{\rm max} (V_i) $, then the analog of  equality~\eqref{imp}
will be true only up to a sign.}
\begin{equation}
\label{imp} \mu_{V_1, V_2} (v_1 \otimes v_2) = \mu_{V_2, V_1}
(v_2 \otimes v_1) \mbox{.}
\end{equation}

For the proof of item~\eqref{it2} is sufficient to show that if $ 1 \le n_1 \le n_2 $, then the central extension
$ \widetilde{GL_ {n_1} (\da_{\Delta})}_{\dr_+^*} $ is obtained from the central extension  $ \widetilde{GL_{n_2} (\da_{\Delta})}_{\dr_+^*} $
by restriction
 to the subgroup $ GL_{n_1} (\da_{\Delta}) $ of the group $ GL_{n_2} (\da_{\Delta}) $ in the image of the map
 $ \theta $. This follows from the construction of the central extension. The proposition is proved.
\end{proof}

\subsection{The case when the set $ \Delta $ is a singleton} \label{dv}
If the set $ \Delta $ is a singleton $ \{x \in C \} $,
then $ \da_{\Delta} = K_{x, C} $ is a finite product of two-dimensional local
fields. The local ring $ \oo_x $ of the  closed point $ x $ on $ X $ and the
completion  $ \hat{\oo}_x $ of this local ring at the maximal ideal
are normal rings without zero divisors\footnote{We recall that we assumed that $ X $ is an integral two-dimensional normal scheme of finite type
over $ \dz $.}. Let $ \eta_C $ be a height $ 1 $ prime
ideal  of the ring $ \oo_x $ which is defined by the curve $ C $. Let
$ \eta_i $ ($ 1 \le i \le m $) be all height $ 1 $ prime ideals of the
ring $ \hat{\oo}_x $ which contain the ideal $ \eta_C \hat{\oo}_x $.
Let $ K_i $ ($ 1 \le i \le m $) be a two-dimensional local field which is
defined as the completion of the field $ \Frac (\hat{\oo}_x) $ over the discrete
valuation associated with the ideal of $ \eta_i $. Then we have
\begin{equation} \label {proiz}
\da_{\Delta} = K_{x, C} = \prod_{i = 1}^m K_i \mbox{.}
\end{equation}

 Let $ K $ be a two-dimensional local field which is
one of the factors in the product~\eqref{proiz}. Let
$ \df_q $ be
  the last residue field of the field $ K $. By considering the corresponding filtrations on the group $ K^n $ and on its subquotients,
  it is not difficult to see that $ \dr_+^* $-torsors which appear in the construction of  central extension~\eqref{ext1}
  (after its restriction to the subgroup $ GL_n (K) $)
 come from  $ q^{\dz} $-torsors. Therefore the restrictions of central extensions~\eqref{ext1} and~\eqref{ext2} to the subgroup
 $ GL_n (K) $ arise from the  elements\footnote{A more precise statement will be obtained later in  proposition~\ref{pr}.} of
 $ H^2 (GL_n (K)), \dz) $
 after applying the map $ \dz \lto \dr_+^* \, : \, a \mapsto q^a $. We note that in contrast to the case of a one-dimensional local field,
 in the case of the two-dimensional local field $ K $ central extensions~\eqref{ext1} and~\eqref{ext2} with $ n> 1 $ are not obtained from $ n = 1 $ using the map
 $ \det \,: \, GL_n (K) \to K^* $.

 \begin{nt}  {\em If $ X $ is surface over a finite field $ \df_q $,
and $ \Delta $ is any subset, then  central  extensions~\eqref{ext1}
and~\eqref{ext2}  are also obtained from the elements of $ H^2 (GL_n
(\da_{\Delta}), \dz) $ as above, since  to construct these central
extensions it is  sufficient to consider the filtrations  by $ \df_q
$-spaces on the space $ \da_{\Delta}^n $ and on its subfactors.
 (See also~\cite{O}.)}  \end{nt}

 We denote  by $ \widehat{GL_n (K)}_{\dr_+^*} $ and by $
\widetilde{GL_n (K)}_{\dr_+^*} $ the restrictions of central
extensions~\eqref{ext1} and~\eqref{ext2} to the subgroup $ GL_n (K)
$ of the group $ GL_n (\da_{\Delta}) $. (We note that from the
construction it follows immediately that if $\da_{\Delta} = \prod_{i
= 1}^m K_i $, then $ GL_n (\da_{\Delta}) = \prod_{i = 1}^m GL_n
(K_i) $, and the central extensions $ \widehat{GL_n
(\da_{\Delta})}_{\dr_+^*} $ and $ \widetilde{GL_n
(\da_{\Delta})}_{\dr_+^*} $ are the Baer sums  on all $ 1 \le i \le
m $ of the central extensions $ \widehat{GL_n (K_i)}_{\dr_+^*} $ and
$ \widetilde{GL_n (K_i)}_{\dr_+^*} $, respectively.)  We  determine
now the element of the group $ \Hom (K_2(K), \dr_+^*) $ which
corresponds to a central extension  $ \widehat{GL_2 (K)}_{\dr_+^*} $
for a two-dimensional local field $ K $ by virtue of the isomorphism
from proposition~\ref{prop-coh}. Let us define a map $ \nu_K (\cdot,
\cdot): K^* \times K^* \lto \dz $ by the following formula
\begin{equation} \label{cel}
\nu_K (f, g) = \nu_{\bar{K}} \left(\overline{\frac{f^{\nu_K (g)}}{g^{\nu_K (f)}}} \right) \mbox{,}
\end{equation}
where $ \nu_K: K^* \to \dz $ and $ \nu_{\bar{K}}: \bar{K}^* \to \dz $ are the
discrete valuations. Since the expression in brackets is the tame symbol without the sign, then $ \nu_K (\cdot, \cdot) $ is a map from
the group $ K_2 (K) $ to  the group $ \dz $.
\begin{prop} \label{pr}
Let $ K $ be  a two-dimensional local field with the  last residue
field $ \df_q $. The central extension $ \widehat{GL_2
(K)}_{\dr_+^*} $ is obtained from the central extension $
\widehat{GL_2 (K)} $ (see  exact sequence~\eqref{cenext}) by means
of the map $ q^{- \nu_K (\cdot, \cdot)} \, : \, K_2 (K) \to \dr_+^*
$.
\end{prop}
\begin{proof}
According to  formula~\eqref{equ}, it is enough to verify that for any $ x, y \in K^* $ the following formula is satisfied
$$
< \diag (y, 1), \, \diag (1, x)> = q^{\nu_K (y, x)} = q^{- \nu_K (x, y)} \mbox{,}
$$
where $ <\cdot, \cdot> $ is the commutator of the  lifting of two commuting elements from the group $ GL_2 (K)  $
to the group $ \widehat{GL_2 (K)}_{\dr_+^*} $.
Indeed, we have
\begin{multline*}
<\diag (y, 1), \, \diag (1, x)> = <\diag (y, 1), \, \diag (x, 1) \diag (x^{-1}, x) > = \\
= <\diag (y, 1), \, \diag (x, 1)> <\diag (y, 1), \, \diag (x^{-1}, \, x)> = \\
= <\diag (y, 1), \, \diag (x^{-1}, \, x)> \mbox{.}
\end{multline*}
Here we used the  bimultiplicativity of the map $ <\cdot, \cdot> $
(see, for example,~\cite[Prop.~6]{O}). In addition, $ <\diag (y, 1),
\, \diag (x, 1)> = 1 $, because  the central extension $
\widehat{GL_2 (K)}_{\dr_+^*} $ splits over the subgroup  $ K^* $ of
the group $ GL_2 (K)_{\dr_+^*} $. Since, by construction, the group
$ \widehat{GL_2 (K)}_{\dr_+^*} $ is a semidirect product, then
\linebreak $ <\diag (y, 1), \, \diag (x^{-1}, \, x)> $ can be
calculated in the group $ \widetilde{GL_2 (K)}_{\dr_+^*} $ (the
answer is the same.) From the construction of  central
extension~\eqref{ext1} and from formula~\eqref{imp} it is clear that
the commutator of the lifting of  diagonal matrices can be
calculated componentwise (for each place on the diagonal). We then
take the product on both components for the answer. Therefore
$$ <\diag (y, 1), \, \diag (x^{-1}, \, x)> = <y, \, x^{-1}> \mbox{,} $$
where $ <y,\, x^{-1}> $ is calculated in the group $
\widetilde{GL_1(K)}_{\dr_+^*} $. Now the formula
\begin{equation} \label{vych}
<y, \, x^{-1}> = q^{\nu_K (y, x)}
\end{equation}
 can be verified by means of the bimultiplicativity of the map $
<\cdot, \, \cdot> $ and by means of the decomposition $ K^* =
t^{\dz} \cdot \oo_K^* $, where $ t $ is a local parameter in the
field $ K $ (compare also with the proof of theorem~1 in~\cite{O}).
The proposition is proved.
\end{proof}

\begin{nt} \label{pro} {\em
It is not difficult to see that for $ n \ge 2 $ the transition from the central extension $ \widetilde{GL_n (K)}_{\dr_+^*} $ to the central extension
$ \widehat{GL_n (K)}_{\dr_+^*} $ (see the construction in section~\ref{constr} above) corresponds, on a level of the second cohomology group, to the projection to the second summand in the right-hand side of formula~\eqref{gen}. More concrete calculations in $ K $-groups will be made during the proof of theorem~\ref{t1} below. }
\end{nt}

\subsection{Central extensions on arithmetic surfaces when  the archimedean valuations are taken into account} \label{ar}
Any number field has the points at infinity: the archimedean valuations. The full (arithmetic)  adelic ring of a number field takes into account these valuations.

 Let $ X $ be a two-dimensional normal integral scheme of finite type over $ \dz $ such that there is a proper surjective morphism to $ \Spec \dz $.
In this case we {\em call} $ X $  as an arithmetic surface.
Let $ X_{\dqq} = X \otimes_{\Spec \dz} \Spec \dqq $ be the generic fiber. For any closed point $ p \in X_{\dqq} $ we define  rings
$$ K_p \widehat{\otimes} \dr = (\dqq (p) \otimes_{\dqq} \dr) ((t_p)), \qquad \oo_{K_p} \widehat{\otimes} \dr = ( \dqq (p) \otimes_{\dqq} \dr) [[t_p]] \qquad \mbox{,} $$
where $ \dqq (p) $ is the residue field of the point $ p $ on the curve $ X_{\dqq} $, and
 $ t_p $ is a local parameter at the point $ p $ on the curve $ X_{\dqq} $ (we note that the local ring of the point $ p $ on the one-dimensional scheme  $ X_{\dqq} $ is a discrete valuation ring). We now define the arithmetic adelic ring $ \da_X^{ar} $
 (see also~\cite[example~11]{OP}) as
 \begin{equation} \label{addecom}
 \da_X^{\rm ar} = \da_X \times \da_{X, \infty} \mbox {,} \qquad \mbox{where
 the ring} \quad \da_{X, \infty} =
\mathop{\prod\nolimits'}_{p \in X_{\dqq}} (K_p \widehat {\otimes}
\dr)
 \end{equation}
is  the restricted product with respect to the subrings $ \oo_{K_p}
\widehat{\otimes} \dr $. We  note that the closed points $ p \in
X_{\dqq} $ are in one-to-one correspondence with the integral
 one-dimensional subschemes of $ X $ such that these subschemes map surjectively
 onto $ \Spec \dz $ (they are the horizontal arithmetic curves on
$ X $). In addition, $ \dqq (p) {\otimes_{\dqq}} \dr = \prod_i L_i $, where
the product is over all equivalence classes of archimedean
valuations of the field $ \dqq (p) $, and the field $ L_i $ is isomorphic either to the field $ \dr $
or to the field $ \dc $. (We note also that the ring $ \da_{X, \infty} $
is a subring in the adelic ring of the curve $ X_{\dr} = X_{\dqq}
\otimes_{\dqq} \dr $  such that the nonzero components of this subring exactly
 correspond  to the components of the adelic ring of the curve $ X_{\dr} $ which come from  the algebraic (not transcendental) points
of the curve $ X_{\dr} $, i.e come from the closed points of the curve $ X_{\dr} $ which are mapped   to the closed points
 under the natural map
$ X_{\dr} \to X_{\dqq} $.) Therefore the definition of the ring $ \da_X^{\rm ar} $
can be interpreted as an addition of the fields
$ \dr((t_p)) $ or $ \dc ((t_p)) $ to the scheme adeles of $ X $. These fields  correspond to the equivalence classes
of archimedean valuations on irreducible horizontal
arithmetic curves and have to satisfy  the adelic condition along the curve
$ X_{\dqq} $.

Let $ \Delta $ be any subset of the set of closed points on $ X_{\dqq} $. Let us define the rings as
$$ \da_{\Delta, \infty} =
\mathop{\prod \nolimits'}_{p \in \Delta} (K_p \widehat{\otimes} \dr)
\qquad \mbox{and} \qquad \oo_{\da_{\Delta, \infty}} = \prod_{p \in
\Delta} (\oo_{K_p} \widehat{\otimes} \dr) \mbox {.}
$$
Similarly to the constructions from section~\ref{constr},   central
extensions\footnote{We note  that the central extension $
\widetilde{GL_n (\da_{\Delta, \infty})}_{\dr_+^*} $ is also obtained
from a central extension of the group $ GL_n (\da_{\Delta, \infty})
$ by the  group $ \dr^* $, which was constructed by E.~Arbarello,
C.~De~Concini and V.~G.~Katz in~\cite{ADK}, by means of the map $
\dr^* \to \dr_+^*: x \mapsto \mid x^{-1} \mid $ which is applied to
the kernel of the central extension. For this  we have to consider
the space $ \da_{\Delta, \infty}^n $ as a vector space over the
field $ \dr $ with the canonical action of the group $ GL_n
(\da_{\Delta, \infty}) $. Then we note that for any
finite-dimensional $ \dr $-vector space $ V $, there is a canonical
identification of $ \dr_+^* $-torsors $ \mid \det (V)^* \setminus 0
\mid $ and $ \mu (V) $ by  integration of differential forms, where
the first $ \dr_+^* $-torsor  is obtained from the $ \dr^* $-torsor
$ \det (V)^* \setminus 0 $ by means of the norm map between the
structure groups of torsors.}  $ \widetilde{GL_n (\da_{\Delta,
\infty})}_{\dr_+^*} $ and $ \widehat{GL_n (\da_{\Delta,
\infty})}_{\dr_+^*} $ of the  group $ GL_n (\da_{\Delta, \infty}) $
by the group $ \dr_+^* $ are defined. From the construction it
follows  the canonical splitting of these central extensions over
the  subgroup  $ GL_n (\oo_{\da_{\Delta, \infty}}) $ of the group $
GL_n (\da_{\Delta, \infty}) $. In addition, for these central
extensions proposition~\ref{prope} and remark~\ref{pro} are
satisfied by similar reasons. An analogue of proposition~\ref{pr} is
formulated as follows.
\begin{prop} \label{pr2}
The central extensions  $ \widehat{GL_2 (\dr ((t)))}_{\dr_+^*} $ and $ \widehat{GL_2 (\dc ((t)))}_{\dr_+^*} $ are obtained from
central extensions~\eqref{cenext} by means of  maps
$$ K_2 (\dr ((t))) \to \dr_+^* \quad : \quad (f, g) \mapsto \left| \overline{\frac{f^{\nu (g)}}{g^{\nu (f)}}} \right|_{\dr} \mbox{,} $$
$$ K_2 (\dc ((t))) \to \dr_+^* \quad : \quad (f, g) \mapsto \left| \overline{\frac{f^{\nu (g)}}{g^{\nu (f)}}} \right|_{\dc}^{2}
\mbox{,}
$$ respectively, where $ \nu $ is  the discrete valuation, $ | \cdot |_{\dr} $ or $ | \cdot |_{\dc} $ are the usual absolute values in the fields
 $ \dr $ or $ \dc $.
\end{prop}
The proof of this proposition is similar to the proof of proposition~\ref{pr}.

We {\em define} a central extension $ \widehat{GL_n (\da_X^{\rm ar})}_{\dr_+^*} $ of  the group $ GL_n (\da_X^{\rm ar}) $  by the
group $ \dr_+^* $ as the Baer sum of
the central extensions  $ \widehat{GL_n (\da_X)}_{\dr_+^*} $ and $ \widehat{GL_n (\da_{X, \infty})}_{\dr_+^*} $ with respect to
decomposition~\eqref{addecom}.
 A central extension  $ \widetilde{GL_n (\da_X^{\rm ar})}_{\dr_+^*} $  is {\em defined}  similarly.

\subsection {Noncommutative reciprocity laws} \label{noncom}
Let $ X $ be a two-dimensional  normal integral scheme of finite
type over $ \dz $. For a closed point $ x \in X $  we {\em define} a
ring $ K_x $ as the localization of the ring $ \hat{\oo}_x $  with
respect to the multiplicative system $ \oo_x \setminus 0 $ (we
recall that $ \oo_x $ and $ \hat{\oo}_x $ are the local ring of the
point $ x $ on $ X $ and its completion by the maximal ideal,
respectively). For any integral one-dimensional subscheme $ C $ on $
X $ we {\em define} a field $ K_C $ as the completion of the field
of rational functions on $ X $ over a discrete valuation which is
given by the curve $ C $. The rings $ K_x $ and $ K_C $  appear
naturally in the theory of two-dimensional adeles when  we describe
the semilocal situation on $ X $, see~\cite{PF}. We have the
diagonal embedding of the ring $ K_x $ into the ring $ \da_X $ via
the embedding in all two-dimensional local fields which arise from
the integral one-dimensional subschemes on $ X $, passing through
the point $ x $. There is also the diagonal embedding of the ring $
K_C $ into the ring $ \da_X $ (or into the ring $ \da_X^{\rm ar} $
if $ X $ is an arithmetic surface) via the embedding  in all
two-dimensional local fields which arise from the points on the
curve $ C $ (here, in the case of an arithmetic surface $ X $ and a
horizontal arithmetic curve $ C $ we  must also embed  the field $
K_C $ into the archimedean part  $ \da_ {X, \infty} $, which takes
into account the archimedean points on $ C $). In addition, the
field of rational functions $ \df_q (X) $ (in the case of a surface
$ X $ over the field $ \df_q $) or the field of rational functions $
\dqq (X) $ (in the case of an arithmetic surface $ X $) are
diagonally embedded into the ring $ \da_X $ or into the ring $
\da_X^{\rm ar} $, respectively. We {\em denote} this field of
rational functions as $ K_X $.

\begin{Th} \label{t1}
Let $ X $ be an integral two-dimensional normal scheme of finite type over $ \dz $ which is either a projective surface
over $ \df_{q} $
or an arithmetic surface. Let $ \da $ be a ring $ \da_X $ (in the case of a surface over $ \df_{q} $) or
be a ring $ \da_{X}^{\rm ar} $
(in the case of an arithmetic surface.) Then the following non-commutative reciprocity laws hold. For any
$ n \ge 1 $ the central extension $ \widehat{GL_n (\da)}_{\dr_+^*} $ of
 the
  group
  $ GL_n (\da) $
   by
    the
     group
      $ \dr_+^* $  splits canonically
 over the following subgroups:
$ GL_n (K_X) $, $ GL_n (K_x) $, $ GL_n (K_C) $, where $ x $ is any closed point on $ X $, and $ C $ is any integral one-dimensional
subscheme of the scheme $ X $.
\end{Th}
\begin{proof}
We first prove the splitting of the central extension  $ \widehat{GL_n (\da)}_{\dr_+^*} $ over the group $ GL_n (K_x) $.

We consider the ring of adeles $ \da_x $ of the scheme $ \Spec {\oo_x} $. This adelic ring
 contains only two-dimensional local fields that come
from the integral one-dimensional subschemes of $ X $ which pass
through the point $ x $. We note that we actually only use the ring
$ \da_x $ for the construction of the restriction of the central
extension to $ GL_n (K_x) $. We consider now the ring of adeles $
\hat{\da} _x $  of the scheme $ \Spec \hat{\oo_x} $. The  ring  $
\hat{\da}_x $ contains the ring $ \da_x $ as a direct factor. But
the ring $ \hat{\da}_x $ also contains  two-dimensional local fields
which are obtained from the height $ 1 $ prime ideals  $ \eta $  of
the ring $ \hat{\oo}_x $ such that $ \eta \cap \oo_x = 0 $. We note
that if $ K_ {\eta} $ is such a two-dimensional local field, then $
K_x \subset \oo_{K_{\eta}} $ (we recall that $ \oo_{K_{\eta}} $ is
the discrete valuation ring of the field $ K_{\eta} $).

Central extensions  $ \widetilde{GL_n (\hat{\da}_x)}_{\dr_+^*} $ and
$ \widehat{GL_n (\hat{\da}_x)}_{\dr_+^*} $ of the group $ GL_n
(\hat{\da}_x) $ by the group $ \dr_+^* $ are constructed  as in
section~\ref{constr}. We first show that the latter central
extension
 splits over the subgroup
$ GL_n (\Frac (\hat{\oo}_x)) $, where $ \Frac (\hat{\oo_x}) $  is
the quotient field of the ring $ \hat{\oo_x} $. It follows from the
construction of the central extension and from the fact that the
group $ SL_n (\Frac (\hat{\oo_x})) $ is perfect  that it is
sufficient to prove\footnote{Indeed, the group $ \widehat{GL_n
(\hat{\da}_x)}_{\dr_+^*} $  is a semidirect product of a group,
which is the restriction of the central extension $ \widetilde{GL_n
(\hat{\da}_x)}_{\dr_+^*} $ to the subgroup $ SL_n (\hat{\da}_x) $,
and of the group $ \hat{\da}_x^* $.   Since the group $ SL_n (\Frac
(\hat{\oo_x})) $ is perfect,  any section of the central extension
over  the  subgroup $ SL_n (\Frac (\hat{\oo_x})) $ is invariant
under the action of an automorphism given by an element of the group
   $ \Frac (\hat{\oo_x})^* $ (we need this automorphism to define
the semidirect product). Therefore this   section together with the
identity  section  over the  group $ \Frac (\hat{\oo_x})^* $ gives
the section  of the  central extension  $ \widehat{GL_n
(\hat{\da}_x)}_{\dr_+^*} $ over the subgroup   $ GL_n (\Frac
(\hat{\oo}_x)) = SL_n (\Frac (\hat{\oo}_x)) \rtimes \Frac
(\hat{\oo_x})^* $.} the  splitting of the central extension  $
\widehat{GL_n (\hat{\da}_x})_{\dr_+^*} $ over the  subgroup $ SL_n
(\Frac(\hat{\oo_x})) $. Let us denote the restriction of the central
extension  $ \widehat{GL_n ( \hat{\da}_x})_{\dr_+^*} $ to the
subgroup $ SL_n (\Frac (\hat{\oo_x})) $ as $ \widehat{SL_n (\Frac
(\hat{\oo_x}))}_{\dr_+^*} = \widetilde{SL_n (\Frac
(\hat{\oo_x}))}_{\dr_+^*} $.  A central extension  $ \widehat{SL
(\Frac (\hat{\oo_x}))}_{\dr_+^*} = \widetilde{SL (\Frac
(\hat{\oo_x}))}_{\dr_+^*} $ of the group $ SL (\Frac (\hat{\oo_x}))
$ by the group  $ \dr_+^* $ is well-defined by the constructed
central extensions. Because of the universality of central
extension~\eqref{uext}, the central extension
 $ \widetilde{SL (\Frac (\hat{\oo_x}))}_{\dr_+^*} $ is
obtained from $ St (\Frac (\hat{\oo_x})) $ by means of some  map $ K_2 (\Frac (\hat{\oo_x})) \to \dr_+^* $.
To calculate this map it is enough to count in the group
$ \widetilde{SL (\Frac (\hat{\oo_x}))}_{\dr_+^*} $ the following expression $ <\diag (u, u^{-1}, 1), \, \diag (v, 1, v^{-1})> $
for any element $ (u, v) \in
K_2 (\Frac (\hat{\oo_x})) $ (see the remark after corollary~11.3 in~\cite{Mi}).

It follows from the construction of the central extension  that the
commutator of the lifting of diagonal matrices can be calculated
separately for each place on the diagonal, and then we take the
product.  We find:
$$
<\diag (u, u^{-1}, 1), \, \diag (v, 1, v^{-1})> = <u,v> \mbox{,}
$$
where $ <u,v> $ is computed in the group $ \widetilde{GL_1 (\Frac (\hat{\oo_x}))}_{\dr_+^*} \subset
 \widetilde{GL_1 (\hat{\da}_x)}_{\dr_+^*} $.
To calculate\footnote{Compare also with the proof of  theorem~2
in~\cite{O}} $ <u,v> $  we decompose the ring $ \hat{\da}_x $ into a direct sum
of two subrings:
\begin{equation} \label{dp}
\hat{\da}_x = \da_1 \oplus \da_2 \mbox{,}
\end{equation}
 where the subring $ \da_1 $ is a finite product of two-dimensional local fields such that for each of these fields the corresponding discrete valuation of at least one of the elements of $ u $ and $ v $ is not equal to zero, and the subring $ \da_2 $  is the adelic product
of all  the rest of  two-dimensional local fields from the ring $ \hat{\da} _x $.
From an analogues of statement~\ref{it1} of proposition~\ref{prope} applied to
decomposition~\eqref{dp} it follows that \linebreak
$ <u,v> = <u,v>_1 \cdot <u,v>_2 $, where
the commutator $ <u,v>_i $ ($ 1 \le i \le 2 $) is computed in the group
$ \widetilde{GL_1 (\da_i})_ {\dr_+^*} $. Since $ u, v \in \oo_{\da_2}^* $
and the central extension $ \widetilde{GL_1 (\da_2})_{\dr_+^*} $
splits over the subgroup $ \oo_{\da_2}^* $, we have that $ <u,v>_2 = 1 $. Again
from an  analogues of statement~\ref{it1} of proposition~\ref{prope} it follows
that $ <u,v>_1 $ is a finite product of commutators such that  each commutator is calculated separately for the two-dimensional local field from the ring
$ \da_1 $.

Now, using calculation~\eqref{vych} from the proof of proposition~\ref{pr},  and also using the equality $ \nu_K (v, u) = 0 $ for elements
$ u, v \in \oo_K^* $, where $ K $ is a two-dimensional local field,
  we obtain that
  $$
  <u,v> = \prod_K q_K^{\nu_K (v, u)} \mbox{,}
  $$
  where $ K $ runs over all two-dimensional local fields which are the  components of the ring $ \hat{\da}_x $,
  and $ q_K $ is the number of elements in the last residue field of the field $ K $.

  The map $ \nu_K (\cdot, \cdot) $ is the composition of the
boundary maps in the Milnor $ K $-theory. Hence the  reciprocity law
around the point $ x $:   $$ \sum_K \nu_K (v, u) \log_{q_x} (q_K) =
0 \mbox{,} $$   where $ q_x $ is the number of elements in the
residue field of the point $ x $, follows from the analogue of the
Gersten-Quillen complex  for Milnor $ K $-theory  (see~\cite[prop.~1
]{Kat}).   Therefore the central extension  $ \widehat{GL_n
(\hat{\da}_x)}_{\dr_+^*} $ splits over the subgroup  $ GL_n (\Frac
(\hat{\oo}_x)) $.

  We note now that it follows from the construction  that the
restriction of the central extension $ \widehat{GL_n
(\da)}_{\dr_+^*} $ to the subgroup $ GL_n (K_x) $   coincides with
the restriction of the central extension $ \widehat{GL_n
(\hat{\da}_x)}_{\dr_+^*} $ to the same subgroup. As we have just
proved,  the latter central extension splits over this subgroup,
because $ GL_n (K_x) \subset GL_n (\Frac (\hat{\oo}_x)) $. Hence,
the central extension $ \widehat{GL_n (\da)}_{\dr_+^*} $ splits over
the subgroup  $ GL_n (K_x) $.

The proof of the splitting of the central extension $ \widehat{GL_n (\da)}_{\dr_+^*} $ over the subgroup
 $ GL_n (K_C) $ is analogous to the previous case. It is reduced to the proof of the splitting of the central
 extension $ \widetilde{SL (K_C)}_{\dr_+^*} $ over the group $ SL (K_C) $. For this purpose, the group $ St (K_C) $
  and the computation of the expression $ <u,v> $ in the group $ \widetilde{GL_1 (K_C)}_{\dr_+^*} $ for
$ (u, v) \in K_2 (K_C) $   are used. From the bimultiplicativity of
the expression $ <\cdot, \cdot> $, from the property $ <t,t> = <-1,
t> $ (which follows from the Steinberg  property) as well as from
the construction of the central extension $ \widetilde{GL_1
(K_C)}_{\dr_+^*} $, it is easy to see that the only non-trivial
 case which is necessary to calculate is $ <a, t> $,
where $ t $ is a local parameter of the field $ K_C $ and $ \nu_t
(a) = 0 $ ($ \nu_t $ is discrete
 valuation in the field $ K_C $). In this case it is easy to see that the element $ <a,t> $ coincides
 with the product over all the normalized absolute values of the field of rational functions on the
 one-dimensional subscheme  $ C $ such that these normalized absolute values are applied to the image
 of the element $ a $ in this field (compare with propositions~\ref{pr} and~\ref{pr2}). Therefore the
 product formula implies $ <a,t> = 1$. Consequently, $ <u,v> = 1 $ for all
$ (u, v) \in K_2 (K_C) $. Thus, we have obtained that the central
extension $ \widehat{GL_n (\da)}_{\dr_+^*} $ splits over the
subgroup  $ GL_n (K_C) $.

The proof of the splitting of the central extension $ \widehat{GL_n
(\da)}_{\dr_+^*} $ over the  subgroup  $ GL_n (K_X) $ is similar to
the case just analyzed. But we have to use the group $ St (K_X) $,
and for the proof  $ <u,v> = 1 $ (where $ (u, v) \in K_2 (K_X) $) we
have to use the product formulas for the normalized absolute values
of the fields of rational functions of a finite number of
one-dimensional integral subschemes on $ X $ (namely, of those
subschemes where the functions $ u $ and $ v $ have zeros or poles).
The theorem is proved.
\end{proof}

\begin{nt}{\em
The splitting of the central extension $ \widetilde{GL_n
(\da)}_{\dr_+^*} $  over the subgroups $ GL_n (K_x) $, $ GL_n (K_C)
$ and $ GL_n (K_X) $ follows immediately from theorem~\ref{t1}. (In
the case of the group $ GL_n (K_x) $, it is necessary first to work
with the central extensions of the group $ GL_n (\Frac
(\hat{\oo}_x)) $, as in the proof of the theorem.) Indeed, according
to remarks~\ref{dec-coh} and~\ref{pro}, the proof is at once reduced
to the case $ n = 1 $. Now from the reciprocity laws  around a point
and along an integral one-dimensional subscheme (we  used these
reciprocity laws above) it follows that the group in which the
commutator is computed and  where we prove the splitting is abelian.
Next, we note that $ \Ext^1_{\dz} (\cdot \ , \dr_+^*) = 0 $, since $
\dr_+^* $ is a divisible group. }
\end{nt}

\section{Unramified Langlands correspondence for two-dimensional local fields} \label{tlf}

\subsection{Unramified Langlands correspondence for one-di\-men\-si\-o\-nal local fields} \label{one-dim}
We first recall the construction for a one-dimensional local field $ K $, see  the surveys~\cite{PR} and~\cite {W}.

Let $ K $ be a one-dimensional local field with a finite residue field $ \df_q $. Let $ K^{\rm nr} $ be the maximal unramified extension of the field
 $ K $.
Then the Galois group $ \Gal (K^{\rm nr} / K) = \Gal (\bar{\df}_q / \df_q) = \hat{\dz} $ is topologically generated by the Frobenius
automorphism $ \rm Fr $.
The reciprocity map is constructed in the following way:
$$
K^* \lto \Gal (K^{\rm nr} / K) \quad : \quad f \longmapsto {\rm Fr}^{\nu_K (f)} \mbox {,}
$$
where $ \nu_K $ is the discrete valuation of the field $ K $.

Now let $ \rho $ be an unramified $ n $-dimensional complex semisimple representation of the  Weil group $ W_K \subset \Gal (K^{\rm ab} / K) $. Then this representation is a semisimple  representation of the group $ \dz $, and up to an isomorphism this representation is determined by a set of $ n $ nonzero complex numbers:
$ \alpha_1, \ldots, \alpha_n $, which are the eigenvalues of the lifting of the Frobenius automorphism $ \rm Fr $.

On the other hand, let us consider the group $ GL_n (K) $. Let $ B \subset GL_n (K) $ be the Borel subgroup of upper triangular matrices.
 We define a character
$ \chi_{\alpha_1, \ldots, \alpha_n}: B \lto \dc^* $ in the following way:
$$
 \chi \left ( \left(
\begin{matrix}
b_1 & * & * & * \\
    & b_2 & * & * \\
    && \cdot & \cdot \\
    &&& b_n
\end{matrix}
\right)
\right) = \alpha_1^{\nu_K (b_1)} \cdot \ldots \cdot \alpha_n^{\nu_K (b_n)} = (q^{- \nu_K (b_1)})^{a_1} \cdot \ldots \cdot
(q^{- \nu_K (b_n)})^{a_n} \mbox {,}
$$
where the element $ (a_1, \ldots, a_n) \in \dc^n / \left (\frac{2
\pi i}{\ln q} \right) \dz^n $ is defined by $ q^{-a_i } = \alpha_i
$, $ 1 \le i \le n $.

For every $ m \ge 1 $ we define a normal subgroup  $ C_m $ of the group $ GL_n (\oo_K) $:
$$ C_m = 1 + t_K ^{m} M_n (\oo_K) \mbox{.} $$
(The expression on the right hand side of this formula is considered in the ring $ n \times n $ matrices $ M_n (\oo_K) $. Here $ \oo_K $ is the discrete valuation ring of the field $ K $, and $ t_K $ is a local parameter.) We denote also $ C_0 = GL_n (\oo_K) $.

We now define  a representation, which is called the principal
series representation of the group $ GL_n (K) $, in the following
way:
$$ \pi (\alpha_1, \ldots , \alpha_n) = \Ind \nolimits^{GL_n (K)}_B (\chi_{\alpha_1, \ldots, \alpha_n}) \mbox{,} $$
where the space of the induced representation $ \Ind^{GL_n (K)}_B (\chi_{\alpha_1, \ldots, \alpha_n}) $ is
$$
\left\{
f: GL_n (K) \to \dc \left|
\begin{array}{l}
 (1) \quad f (bg) = \chi_{\alpha_1, \ldots, \alpha_n} (b) f (g) \mbox{,} \quad \forall b \in B \mbox{,} \quad \forall g \in GL_n (K) \\
 (2) \quad \exists n_f \ge 1 \; : \; f (hu) = f (h) \mbox{,} \quad \forall h \in GL_n (K) \mbox{,} \quad \forall u \in C_{n_f}
\end{array}
\right.
\right\} \mbox{.}
$$
The group $ GL_n (K) $ acts on this space by right translations, i.e. for any $ g \in GL_n (K) $ and  for any
$ f \in \Ind^{GL_n (K)}_B (\chi_{\alpha_1, \ldots, \alpha_n}) $ we have $ (gf) (x) = f (xg) $ for any element $ x \in GL_n (K) $.

\begin{nt} {\em
We  used  the non-normalized induction in the  definition of the representation $ \pi (\alpha_1, \ldots ,\alpha_n) $.
Usually, for the normalized induction one adds  to the character $ \chi_{\alpha_1, \ldots, \alpha_n} $ a character which  takes into account the non-unimodularity of the group $ B $.}
\end{nt}

\begin{nt} \label{dejst} {\em
The definition of the induced representation can be rewritten in the following way (this will be important in what follows.)
 We consider for every $ m \ge 0 $ an infinite-dimensional  $ \dc $-vector space $ V_m = \prod_{g \in GL_n (K) / C_m} \dc $,
 where the vectors from $ V_m $ are any functions $ f $ from the set of left cosets  $ GL_n (K) / C_m $ to the field $ \dc $.
Then the group $ B $ acts on the space $ V_m $ in the following way:
$$
b (f) (x) = \chi_{\alpha_1, \ldots, \alpha_n} (b) f (b^{-1} x) \mbox{,}
$$
where $ b \in B $, $ f \in V_m $, $ x \in GL_n (K) / C_m $. It is now clear that
$$
\pi (\alpha_1, \ldots , \alpha_n) =
 \bigcup_{m \ge 0} V_m^{B} \mbox{,}
$$
where $ V_m^{B} $ is the space of $B$-invariant elements. The group $ GL_n (K) $ acts on
the space $ \pi (\alpha_1, \ldots , \alpha_n) $ as follows: $ g (f) (x) = f (xg) $, where $ f \in V_m^ B $,
$ g, x \in GL_n (K) $, $ g (f) \in V_l $ for some $ l \ge 0 $ (which depends on $ m \ge 0 $ and $ g $).
}
\end{nt}

A representation  of the group $ GL_n (K) $ on a vector space $ W $ is called {\em smooth} if \linebreak $ W = \bigcup_{n \ge 0} W^{C_n} \mbox{,} $.

An irreducible representation of the group $ GL_n (K) $ on a vector space $ W $  is called {\em
spherical} (or {\em unramified}) if
$ W^{GL_n (\oo_K)} \ne
0$.

We note that by construction, $ \pi (\alpha_1, \ldots , \alpha_n) $ is a  smooth representation of the group $ GL_n (K) $. In addition,
$$ \dim_{\dc} \pi (\alpha_1, \ldots , \alpha_n)^{GL_n (\oo_K)} = 1 \mbox{.} $$

The representation  $ \pi (\alpha_1, \ldots , \alpha_n) $ of the
group $ GL_n (K) $  always admits  a unique spherical subquotient
such that this subquotient goes into an isomorphic subquotient under
a permutation of the complex numbers $ \alpha_1, \ldots, \alpha_n $.
(Moreover, it is always possible  to permute the complex numbers $
\alpha_1, \ldots, \alpha_n $ so that this subquotient will be a
quotient.) Now the unramified Langlands correspondence associates
the  representation $ \rho $ (see the beginning of this section)
with the above quotient representation of the group $ GL_n (K) $.
And all the spherical admissible\footnote{We don't recall here the
definition of an admissible representation.} representations of the
group $ GL_n (K) $ are  obtained in this way.

We note that the irreducible admissible representations of the group
$ GL_n (K) $  are exactly the  irreducible smooth representations of
the same group, see~\cite[Theorem~3.25]{BZ}. Therefore in the
preceding paragraph,  the spherical admissible representations of
the group $ GL_n (K) $ can be replaced   by  the spherical smooth
representations of the same group.

\subsection{An action of a group on a $ k $-linear category} \label {gract}
We recall the   definition of an action of a group $ G $ on a $ k $-linear category $ \cal B $, where $ k $ is a field (see
also~\cite{GK}).
\begin{defin} \label{catact}
An action of a  group $ G $ on a $ k $-linear category $ \cal B $ consists of the following data.
\begin{enumerate}
\item For every element $ g \in G $ there is a $ k $-linear functor $ \tau (g) \, : \, \cal B \to \cal B $.
\item For every pair of elements $ g, h \in G $ there is an isomorphism of functors $ \psi_{g, h} \, : \, \tau (g) \circ \tau (h)
\Rightarrow \tau (gh) $.
\item There is an isomorphism of functors $ \psi_1 \, : \, \tau (1) \Rightarrow \Id_{\cal B} $.
\end{enumerate}
Besides, the following conditions are satisfied.
\begin{enumerate}
\item For any $ g, h, k \in G $  the associativity holds: \\
$ \psi_{gh, k} (\psi_{g, h} \circ \tau (k)) =  \psi_{g, hk} (\tau (g) \circ \psi_{h, k}) $.
\item For any $ g \in G $   formulas $ \psi_{1, g} = \psi_1 \circ \tau (g) $ and $ \psi_{g, 1} = \tau (g) \circ \psi ( 1) $ hold.
\end{enumerate}
\end{defin}

If it is necessary to specify the category which we use,  we will use  notations $ \tau_{\cal B} (g) $, $ \psi_{{\cal B}, g, h} $ and
$ \psi_{{\cal B}, 1} $.

\begin{defin}
Let $ \cal B $ and $ \cal C $ be $ k $-linear categories with a $ G $-action. A $ G $-linear functor
$ \mathfrak{T} = (T, \ve) $ from $ \BB $ to $ \CC $ consists of the following data.
\begin{enumerate}
\item There is a $ k $-linear functor $ T \,: \, \BB \to \CC $.
\item For any element $ g \in G $ there is an isomorphism of functors \linebreak
$ \ve_g \,: \, T \circ \tau_{\BB} (g) \Rightarrow \tau_{\CC} (g) \circ T $.
\end{enumerate}
Besides, the following conditions are satisfied.
\begin{enumerate}
\item For any elements $ g, h \in G $ the following diagram is commutative:
\begin{diagram}
T \circ \tau_{\BB}(g)  \circ \tau_{\BB}(h)  &&
\rImplies^{ T \circ  \psi_{\BB, g, h} }
& & T \circ \tau_{\BB} (g h)  \\
\dImplies^{\ve_g \circ \tau_{\BB}(h)}
& & & &
\dImplies_{\ve_{gh}}  \\
\tau_{\CC}(g) \circ T \circ \tau_{\BB}(h) &  \rImplies_{\tau_{\CC}(g) \circ \ve_h}  & \tau_{\CC}(g) \circ \tau_{\CC}(h) \circ T &
\rImplies_{\psi_{\CC,g,h} \circ T}
&  \tau_{\CC}(g h) \circ T\\
\end{diagram}
\item A formula $ (\psi_{\CC, 1} \circ T) \, \ve_1 = T \circ \psi_{\BB, 1} $ holds.
\end{enumerate}
\end {defin}

\begin{nt} \label{G-equiv}
{\em
$ G $-linear functors from $ \BB $ to $ \CC $  are all objects of a natural $ k $-linear category, which we denote as
$ {\mathcal Hom}_G (\BB, \CC) $. If $ (L, \ve) $ and $ (M, \epsilon) $ are objects of the  category $ {\mathcal Hom}_G (\BB, \CC)) $, then, by definition, the $ k $-vector space
$ \Hom_{{\mathcal Hom}_G (\BB, \CC)} ((L, \ve), (M, \epsilon)) $ consists of morphisms of functors $ \varphi: L \Rightarrow M $ such that the following diagram
is commutative  for any element  $ g \in G $:
\begin{diagram}
L \circ \tau_{\BB}(g) &\rImplies^{\ve_{g}} &  \tau_{\CC}(g) \circ L \\
\dImplies^{\varphi \circ \tau_{\BB}(g)} & &\dImplies_{\tau_{\CC}(g) \circ \varphi}\\
M \circ \tau_{\BB}(g) &\rImplies_{\epsilon_{g}} & \tau_{\CC}(g) \circ M
\end{diagram}

If $ \BB $, $ \CC $, $ \cal D $ are $ k $-linear categories with a $ G $-action, and
$ (L, \ve) \in \Ob ({\mathcal Hom}_G (\BB, \CC))) $,
 $ (N, \zeta) \in \Ob ({\mathcal Hom}_G (\CC, \cal D))) $, then
 $$
 (N, \zeta) \circ (L, \ve) = (N \circ L , \eta) \in \Ob ({\mathcal Hom}_G (\BB, \cal D))) \mbox{,}
 $$
 where for any $ g \in G $ the morphism of
functors $ \eta_g \,: \, N \circ L \circ \tau_{\BB} (g) \Rightarrow \tau_{\cal D} (g) \circ N \circ L $
 is defined as the composition of morphisms of functors
$$
 N \circ L \circ \tau_{\BB}(g) \stackrel{N \circ \ve_g}{\Longrightarrow} N \circ \tau_{\CC}(g) \circ L
 \stackrel{\zeta_g  \circ L}{\Longrightarrow} \tau_{\cal D}(g) \circ
 N \circ L \mbox{.}
 $$

Now we can say that  $ k $-linear categories $ \BB $ and $ \CC $ with a $ G $-action are {\em $ G $-equivalent} if there exists a $ G $-linear functor
$ {\mathfrak L} $ from $ \BB $ to $ \CC $ and there exists a $ G $-linear functor $ {\mathfrak N} $ from $ \CC $ to $ \BB $ such that the $ G $-linear functors $ \mathfrak{N} \circ \mathfrak{L} $ and $ (\Id_{\BB}, \mbox{id}) $ are isomorphic as objects of the category
$ {\mathcal Hom}_G (\BB, \BB) $,
and  the $ G $-linear functors $ \mathfrak{L} \circ \mathfrak{N} $ and $ (\Id_{\CC}, \mbox{id}) $ are isomorphic as objects of the category
$ {\mathcal Hom}_G ( \CC, \CC) $.

It is not difficult to understand that to determine that the categories $ \BB $ and $ \CC $ are $ G $-equivalent, it is enough   to construct only a
$ G $-linear functor
$ (L, \ve): \BB \to \CC $ and
to check that the functor $ L $ is an equivalence of categories.
}
\end{nt}

We now define a category of $ G $-equivariant objects.
\begin{defin}
Let $ \BB $ be a $ k $-linear category with a $ G $-action. A category of $ G $-equivariant objects $ \BB^G $ is defined as follows.
\begin{enumerate}
\item An object of the category $ \BB^{G} $ consists of an object $ E $  of the category $ \BB $ and of a system of isomorphisms:
$$
\theta_g \; : \; E \lto \tau (g) (E) \quad \mbox{for any} \quad g \in G \mbox{,}
$$
such that the following diagram is commutative  for any elements $ g, h \in G $:
\begin{equation}  \label{equidiag}
\begin{diagram}
E &\rTo^{\theta_g} &  \tau(g)(E)\\
\dTo^{\theta_{gh}} & &\dTo_{\tau(g)(\theta_h)}\\
\tau(gh)(E) &\lTo_{\psi_{g,h,E}} & \tau(g) (\tau(h)(E))
\end{diagram}
\end{equation}
\item A morphism of equivariant objects  $ \,: \, (E, \theta) \lto (D, \vartheta) $ is a morphism $ f \in \Hom_{\BB} (E, D) $
such that  $ \tau (g) (f) \, \theta_g = \vartheta_g f $.
\end {enumerate}
\end {defin}

\begin{nt}  {\em
 If $ (E, \theta) \in \Ob (\BB^G) $, then any object $ D $ isomorphic to $E$  has an evident  $ G $-equivariant structure $ \vartheta $.}
\end{nt}

\begin{nt}  \label{zam}  {\em
Let $ (E, \theta) $ and $ (D, \vartheta) $ be objects of the category $ \BB^{G} $. Then
the group $ G $ acts
on the $ k $-vector space
$ \Hom_{\BB} (E, D) $  in the following way: $ f \mapsto \vartheta_g^{-1} \tau (g) (f) \, \theta_g $. (From
diagram~\eqref{equidiag} it follows  that this action is  a group action.) Now it is  easy to obtain that
$
 \Hom \nolimits_{\BB} (E, D)^G = \Hom \nolimits_{\BB^G} ((E, \theta), (D, \vartheta)) 
$.}
\end{nt}

\subsection{A categorical analogue of the unramified principal series representations} \label{catan}
Let now (and until the end of the article) $ K $ be a two-dimensional local field with the last finite residue field $ \df_q $.
Let $ \oo_{K} $ be the  rank $ 1 $ discrete valuation ring of the  two-dimensional local field $ K $, and $ t_K $ be a local parameter relative to this discrete valuation. By analogy with
section~\ref{one-dim} we define for every $ m \ge 1 $ a  normal (congruence) subgroup of the group $ GL_l (\oo_K) $ ($ l \ge 1 $):
$$
{\mathbb C}_m = 1 + t_K^{m} M_l (\oo_K) \mbox {.}
$$
We  denote also $ \CCC_0 = GL_l (\oo_K) $.

\begin{defin}
We say that a category $ \BB $  is a generalized $ 2 $-vector space over a field $ k $, if $ \BB $ is a $ k $-linear abelian category. We say that the category $ \BB $ has finite dimension $ r $, if the category  $ \BB $ is equivalent to the category $ (\Vect_k^{\rm fin})^r $.
\end{defin}

Let us  introduce a categorical analogue of a smooth representation.
\begin{defin} \label{glad}
An  action of the group $ GL_l (K) $ on a generalized $ 2 $-vector space $ \BB $ is called smooth if the following condition is satisfied.
  For any objects $ E $, $ D $ from $ \BB $ and for any morphism $ f \in \Hom\nolimits_{\BB} (E, D) $ there are an integer $ m \ge 0 $, objects
    $ (E, \theta) $ and $ (D, \vartheta) $ from the category $ \BB^{{\mathbb{C}}_m} $ such that  $ f \in \Hom_{\BB^{{\mathbb{C}}_m}}
((E, \theta),
    (D, \vartheta)) $.
\end{defin}

In particular, if we consider in this definition  the identity morphism, then  we see that on any object  $ \BB $ there is a
$ {{\mathbb{C}}_m} $-equivariant structure for some $ m \ge 0 $.

\medskip

For a two-dimensional local field $ K $ with the last residue field $ \df_q $ let $ K^{\rm nr} $ will be its maximal unramified extension as a two-dimensional local field. Namely, the field $ K^{\rm nr} $ is unramified relative to the discrete valuation of the field $ K $, and the residue field $ \overline{K^{\rm nr}} $ of the field
 field $ K^{\rm nr} $ is  unramified relative to the discrete valuation of the residue field $ \bar{K} $ of the  field $ K $. Then the Galois group
 $ \Gal (K^{\rm nr} / K) = \hat{\dz} $ is topologically generated by the Frobenius automorphism $ \rm Fr $. The two-dimensional unramified reciprocity map is defined in the following way (see~\cite[theorem~1]{P}):
 \begin{equation} \label{vzaim}
 K_2 (K) \lto \Gal (K^{\rm nr} / K) \quad : \quad (f, g) \longmapsto {\rm Fr}^{\nu_K (f, g)} \mbox{, }
 \end{equation}
where the map $ \nu_K (\cdot, \cdot) $ was defined by formula~\eqref{cel}.

 We {\em define} a Weil group\footnote{It follows from the two-dimensional local
class field  theory  that the group $ W_K $  is a dense subgroup of the
profinite group $ \Gal (K^{\rm ab} / K) $. The image of the group $ W_K $ under the
map  $ \Gal (K^{\rm ab} / K) \to \Gal (K^{\rm nr} / K) =
\Gal (\bar{\df}_q / \df_q) = \hat{\dz} $ is the group $ \dz $.}
 $ W_K $  of a two-dimensional local field $ K $  as the group which is the image of the whole reciprocity map
 $ K_2 (K) \to \Gal (K^{\rm ab} / K) $. Then any $ n $-dimensional unramified semisimple complex representation $ \rho $ of the Weil group
$ W_K $ pass through the group $ \dz $. Therefore the representation $\rho$ is determined up to an isomorphism by the set of $ n $ nonzero complex numbers $ \alpha_1, \ldots, \alpha_n $. Now an element
$ (a_1, \ldots, a_n) \in \dc^n / \left(\frac{2 \pi i}{\ln q} \right) \dz^n $ is defined in the following way:
$ q^{-a_i } = \alpha_i $, $ 1 \le i \le n $.

Now we consider  an arbitrary {\em ordered} set $ (a_1, \ldots, a_n) $ as an element of $ \dc^n / \left(\frac{2 \pi i}{\ln q} \right) \dz^n $. Let $ A $ be an arbitrary $ \dr_+^* $-torsor, $ b $ be any element from $ \dc $. We consider a one-dimensional $ \dc $-vector space
$$ A^b = (A \otimes_{\dr_+^*} \dc^*) \cup 0 \mbox{,} $$
where $ \dc^* $-torsor $ A \otimes_{\dr_+^*} \dc^* $ is defined by means of a homomorphism
$$ \dr_+^* \lto \dc^* \; : \; x \longmapsto x^{b} \mbox{.} $$

We recall that in  \S~\ref{C2} we have constructed the central extension
\begin{equation} \label{GL}
1 \lto \dr_+^* \lto \widehat{GL_2 (K)}_{\dr_+^*} \stackrel{\pi}{\lto} GL_2 (K) \lto 1 \mbox{.}
\end{equation}
We note that it follows from the construction of this central extension that there exists its canonical splitting over
the subgroup $ GL_2 (\oo_K) $.
For any element $ g \in GL_2 (K) $ we define a $ \dr_+^* $-torsor $ A_g = \pi^{-1} (g) $.
Then, by means of central extension~\eqref{GL}, for any element
$ b \in \dc $ and for any elements $ g, h \in GL_2 (K) $  we obtain a canonical isomorphism
of one-dimensional $ \dc $-vector spaces
\begin{equation} \label{iso}
(A_{g})^b \otimes_{\dc} (A_{h})^b \lto (A_{gh})^b \mbox{,}
\end{equation}
which satisfies  associativity for any elements $ g, h, k $ of the group $ GL_2 (K) $.

For any integer $ n \ge 1 $, in the group $ GL_{2n} (K) $ we consider the standard pa\-ra\-bo\-lic subgroup $ P $  defined by block upper triangular matrices:
$$
P =\left\{ g=\left(
\begin{matrix}
g_1 & *    & *      & * \\
    & g_2  & *      & * \\
    &      & \cdot  & \cdot \\
    &      &        &  g_n
\end{matrix}
\right) \; : \; g_i \in GL_2(K)
\right\}  \mbox{.}
$$
 By an element $ (a_1, \ldots, a_n) \in \dc^n / \left(\frac{2 \pi i}{\ln q} \right) \dz^n $ we {\em define} an action of  the group $ P $
 on the $ \dc $-linear category
$ \Vect_{\dc}^{\rm fin} $ as follows. To do this, for any element $ g \in P $ we define a $ \dc $-linear functor
$ \tau_{a_1, \ldots, a_n} (g) \,: \, \Vect_{\dc}^{\rm fin} \to \Vect_{\dc}^{\rm fin} $ in the following way:
\begin{equation} \label{act}
\tau_{a_1 , \ldots, a_n}(g) (Y) = (A_{g_1})^{a_1} \otimes_{\dc} \ldots \otimes_{\dc} (A_{g_n})^{a_n} \otimes_{\dc} Y  \; \mbox{,} \quad \mbox{ где} \quad
 Y \in \Vect\nolimits_{\dc}^{\rm fin} \mbox{.}
 \end{equation}
    From maps~\eqref{iso}, isomorphisms of functors $ \tau_{a_1, \ldots, a_n} (g) \circ \tau_{a_1, \ldots, a_n} (h) \Rightarrow \tau_{a_1, \ldots, a_n} (gh) $ (where
$ g, h \in P $)
and $ \tau_{a_1, \ldots, a_n} (1) \Rightarrow \Id_{\Vect_{\dc}^{\rm fin}} $ are  obtained in an obvious way. And  for these isomorphisms  all the conditions from  definition~\ref{catact} are satisfied.
\begin{nt} {\em
 The constructed action of the group $ P $ on the category $ \Vect_{\dc}^{\rm fin} $ really depends only on an element of the
 group $ \dc^n / \left(\frac{2 \pi i}{\ln q } \right) \dz^n $
(but not on an element of the group $ \dc^n $), because   central
extension~\eqref{GL} comes from an element of $ H^2 (GL_2 (K), \dz) $ after the application of the homomorphism
$ \dz \to \dc^* \,: \, a \mapsto q^a $ (see \S~\ref{dv}).}
\end{nt}

Now we {\em construct} a categorical analogue $ \VV_{a_1, \ldots, a_n} $ of the principal series representations  of the group $ GL_{2n} (K) $.
For every $ m \ge 0 $ we consider
a $ \dc $-linear category
$$ \VV_m = \prod_{g \in GL_{2n} (K) / \CCC_m} \Vect\nolimits_{\dc}^{\rm fin}  $$
whose objects are all the maps $ f $ from a set of left cosets  $ GL_{2n} (K) / \CCC_m $ to objects of the category
$ \Vect_{\dc}^{\rm fin} $, and the morphisms in the category $ \VV_m $ are obvious.
  There is an action $ \xi_{a_1, \ldots, a_n} $ of the group $ P $
  on the category $ \VV_m $:
\begin{equation} \label{catde}
\xi_{a_1, \ldots, a_n} (p) (f) (x) = \tau_{a_1, \ldots, a_n} (p) (f (p^{-1} x)) \mbox{,}
\end{equation}
 where $ p \in P $, $ x \in GL_{2n} (K) / \CCC_m $, and $ f \in \Ob (\VV_m) $. (This action is clearly  defined also on the
morphisms in $ \VV_m $.)
We 
note that for any integers $ m_2 \ge m_1 \ge 0 $ there is a natural $ P $-linear functor $ \VV_ {m_1} \to \VV_ {m_2} $ which arises from the map  of sets
 $ GL_{2n} (K) / \CCC_{m_2} \rightarrow GL_{2n} (K) / \CCC_{m_1} $. This functor induces the natural functor $ Q_{m_1, m_2} \,: \,
 (\VV_{m_1})^{P} \to (\VV_{m_2})^{P} $
 between the categories of $ P $-equivariant objects. (We note that we have also here  a strict equality of functors
 $ Q_{m_2, m_3} \circ Q_{m_1, m_2} = Q_{m_1, m_3} $ for any $ 0 \le m_1 \le m_2 \le m_3 $). Now we  define a $ \dc $-linear
 category:
\begin{equation}  \label{2limform}
 \VV_{a_1, \ldots , a_n} = \mathop{\underrightarrow\LLim}_{m \ge 0} \, (\VV_m)^P   \mbox{,}
 \end{equation}
where the objects of the category $ \VV_{a_1, \ldots, a_n} $ are  all the pairs $ (m, f) $ (an integer $ m \ge 0 $ and
$ f \in \Ob ((\VV_{m} )^{P}) $), and the morphisms are defined as
$$
\Hom\nolimits_{\VV_{a_1, \ldots , a_n} } ((m_1,f_1), (m_2,f_2))  = \mathop{\underrightarrow\Lim}_{m \ge \max(m_1,  m_2)} \Hom\nolimits_{(\VV_{m})^{P}} (Q_{m_1, m}(f_1), Q_{m_2, m}(f_2)) \mbox{.}
$$
(The definition of the categorical direct limit $ \mathop{\underrightarrow\LLim} $, which is used in  formula~\eqref{2limform},
is taken from~\cite[Appendix]{K2}.)

\begin{Th} \label{th2} We fix an element $ (a_1, \ldots, a_n) \in \dc^n / \left(\frac{2 \pi i}{\ln q} \right) \dz^n $. The following properties are satisfied.
\begin{enumerate}
\item \label{st1} The  category $ \VV_{a_1, \ldots, a_n} $ is a generalized $ 2 $-vector space over the field $ \dc $. There is a natural smooth action of the group $ GL_{2n} (K) $ on the category $ \VV_{a_1, \ldots, a_n} $.
\item \label{st2} The generalized $ 2 $-vector space $ (\VV_{a_1, \ldots, a_n})^{GL_{2n} (\oo_K)} $ contains as a full subcategory a category of dimension $ 1 $ (i.e a category which is equivalent to the category $ \Vect_{\dc}^{\rm fin} $).
\item \label{st3}
Let $ n = 1 $. We consider an action of the group $ GL_2 (K) $ on the category $ \Vect_{\dc}^{\rm fin} $ such that this action  depends on a parameter
$ a_1 \in \dc / \left(\frac{2 \pi i}{\ln q} \right) \dz $ and arises from an  unramified character of the group $ \Gal (K^{\rm sep} / K) $ (as it is described by formula~\eqref{vzaim} and in \S~\ref{abcase}). There is a $ GL_2 (K) $-linear functor $ \mathfrak{T} = (T, \ve) $ from the category  $ \Vect_{\dc}^{\rm fin} $ to the category $ \VV_{a_1} $ such that the functor $ T $ is fully faithful.
\end{enumerate}
\end{Th}
\begin{nt} {\em Functors from the category $ \Vect_{\dc}^{\rm fin} $, which arise in items~\ref{st2} and~\ref{st3} of theorem~\ref{th2},
are automatically exact functors, because the category  $ \Vect_{\dc}^{\rm fin} $ is semisimple.}
\end{nt}
\begin{proof} (of theorem~\ref{th2}).
 The category $ \VV_{a_1, \ldots, a_n} $ is abelian by the following reasons. Since a category
 $ \prod_{g \in X} \Vect_{\dc}^{\rm fin} $ is abelian, where $ X $ is any set, we have that the category of $ P $-equivariant objects
  $ (\VV_m)^P $  is abelian  for any $ m \ge 0 $. In addition, $ \mathop{\underrightarrow \LLim} $ of a set of abelian categories with exact transition  functors  is an abelian category.

 An action  $ \sigma $ of the group $ GL_{2n} (K) $ on the category $ \VV_{a_1, \ldots, a_n} $ is  defined as follows. (This action is a categorical analogue of the corresponding action for one-dimensional local fields, see remark~\ref{dejst}.)
 Let
$$ (M, \tilde{f}) \in \Ob (\VV_{a_1, \ldots, a_n}) \quad \mbox{,} \quad g \in GL_{2n} (K) \mbox{.} $$
Taking into account the map $ GL_{2n} (K) \to GL_{2n} (K) / \CCC_m $, we consider $ \tilde{f} = (f, \theta) $, where $ f $ is a map  from the set $ GL_{2n} (K) $ to objects of the
category $ \Vect_{\dc}^{\rm fin} $, and $ \theta = \{\theta_{p, x} \} $ ($ p \in P $, $ x \in GL_ {2n} ( K) $) defines the structure of a $ P $-equivariant object on $ f $ under the action of the group $ P $ by  formula~\eqref{catde}.
Then, by definition, we put $ \sigma (g) (m, (f, \theta)) = (l, (g (f), g (\theta))) $, where $ g (\{\theta_{ p, x} \}) =
\{\theta_{p, xg} \} $, $ g (f) (x) = f (xg) $ for any $ x \in GL_{2n} (K) $, and an integer $ l \ge 0 $  is the smallest non-negative integer with the property  $ g \CCC_m g^{-1} \supset \CCC_l $. (The existence of such an integer $ l $ follows, for example, from the Cartan decomposition of
the group  $ GL_{2n} (K) $
over the discrete valuation field $ K $.)  Here the map $ g (f) $ can be regarded as a well-defined map from
the set $ GL_{2n} (K) / \CCC_l $ to objects of the category $ \Vect_{\dc}^{\rm fin} $, and $ g (\theta) $ defines the structure of a
$ P $-equivariant object on $ g (f) $. Evidently,  an action of the functor $ \sigma (g) $ on the morphisms of the category
$ \VV_{a_1, \ldots, a_n} $ is also defined. It is clear that we obtained the well-defined action $ \sigma $ of the group $ GL_{2n} (K) $ on the category $ \VV_{a_1, \ldots, a_n} $.

It follows at once from definition~\ref{glad} and from the construction of the category $ \VV_{a_1, \ldots, a_n} $ as
a $ 2 $-inductive limit
  that the action
$ \sigma $ is smooth. (For any $ m \ge 0 $ the group $ \CCC_m $ acts identically on the category $ (\VV_m)^ P $. Therefore  on any object of this category there is always the $ \CCC_m $-equivariant structure given by the identity morphisms.) Thus we have proved assertion~\ref{st1} of the theorem.

We construct a functor $ L $ from the category $ \Vect_{\dc}^{\rm fin} $ to the category
\linebreak
 $ (\VV_{a_1, \ldots, a_n})^{GL_{2n} (\oo_K)} $.
We note that it follows from the Iwasawa decomposition of the group $ GL_{2n} (K) $ over the discrete valuation field $ K $
that
 $ GL_{2n} (K) = P \cdot GL_{2n} (\oo_K) $.
For each coset
$ x \in GL_{2n} (K) / GL_{2n} (\oo_K) $ we fix
its representative $ g_x \in GL_{2n} (K) $ and fix some decomposition  $ g_x = p_x h_x $, where $ p_x \in P $,
$ H_x \in GL_{2n} (\oo_K) $. Let $ V \in \Ob (\Vect_{\dc}^{\rm fin}) $. Then we define $ f_V \in \Ob (\VV_0) $ as
$ f_V (x) = \tau_{a_1, \ldots, a_n} (p_x) (V) $
(see formula~\eqref{act}).
There is a canonical splitting of central extension~\eqref{GL} over the  subgroup $ GL_2 (\oo_K) $. This splitting gives a canonical trivialization of the action $ \tau_{a_1, \ldots a_n} $ on the category $ \Vect_{\dc}^{\rm fin} $ when this action is restricted to the group
$ P \cap GL_{2n} ( \oo_K) $. This trivialization gives a unique lift of the object $ f_V $  to the object $ \widetilde{f_V} $ of
 the category  $ (\VV_0)^ P $. Since the group $ GL_{2n} (\oo_K) $ acts identically on the category $ (\VV_0)^P $, there is an obvious lift
of the object $ (0, \widetilde{f_V}) $ from the category $ \VV_{a_1, \ldots, a_n} $ to the object $ L (V) $ of the category
$ (\VV_{a_1, \ldots, a_n})^{GL_{2n} (\oo_K)} $
by means of  the identity morphisms.
  Evidently, $ L $ is defined on morphisms of the category $ \Vect_{\dc}^{\rm fin} $. It is not difficult to see that $ L $  a fully faithful functor from the category $ \Vect_{\dc}^{\rm fin} $ to the abelian category $ (\VV_{a_1, \ldots, a_n})^{GL_{ 2n} (\oo_K)} $.
Assertion~\ref{st2} of the theorem is proved.

A functor $ T $ from assertion~\ref{st3} of the theorem is the composition of the newly constructed functor $ L $ (see the proof
of assertion~\ref{st2} of this theorem) and of the forgetful functor (which   forgets  the $ GL_2 (\oo_K) $-equivariant structure).
 Namely, from  $ V \in \Ob (\Vect_{\dc}^{\rm fin}) $ we construct $ T (V) = (0, \widetilde{f_V}) \in \Ob (\VV_{ a_1}) $. We note that the functor
 $ T $ is fully faithful by the following reasons.   Firstly,  the functor $ L $ is fully faithful. Secondly, since the
$ GL_2 (\oo_K) $-equivariant structure was given by the identity morphisms,   the functor which forgets   such a structure  is also
fully faithful.

 It follows from the description of the reciprocity map for the two-dimensional local field $ K $ (see formula~\eqref{vzaim}),
from the construction from~\S~\ref{abcase} and from  proposition~\ref{pr}  that the action of the  group $ GL_2 (K) $
on the category $ \Vect_{\dc}^{\rm fin} $ such that this action depends on a parameter
$ a_1 \in \dc / \left(\frac{2 \pi i} {\ln q} \right) \dz $ and arises from an unramified character  of the group
$ \Gal (K^{\rm sep} / K) $, is described by the functors $ \tau_{a_1} (g) $ by  formula~\eqref{act}, where $ g \in P = GL_2 (K) $.
 Now
 a $ GL_2 (K) $-linear structure  $ \ve $ on the functor $ T $ is constructed by means of the canonical splitting of
 central extension~\eqref{GL} over the subgroup of $ GL_2 (\oo_K) $.
The theorem is proved.
\end{proof}

\begin{nt} {\em
In items~\ref{st2} and~\ref{st3} of theorem~\ref{th2} we do not obtain a statement on an equivalence of categories by the following reason. For example,
in item~\ref{st2} we use that the group $ GL_{2n} (\oo_K) $ acts identically on the category $ (\VV_0)^P $, and we take the identity morphisms in
the $ GL_{2n} (\oo_K) $-equivariant structure. In this case, besides  the $ GL_{2n} (\oo_K) $-equivariant structure defined by the identity morphisms, there are many other $ GL_{2n} (\oo_K) $-equivariant structures on the objects of the category $ (\VV_0)^P $  such that these structures are non-isomorphic to the identity structure.
In item~\ref{st3} we have similar problems, because  $ P \cap \CCC_m \ne \{1 \} $ for every $ m \ge 0 $.
}
\end{nt}

\begin{nt} {\em
The category $ \VV_{a_1, \ldots, a_n} $ with the $ GL_{2n} (K) $-action is a categorical analogue of the induced representation
from \S~\ref{one-dim}.
For finite groups $ H \subset G $ and a category $ \BB $ with an $H$-action, a category of the induced $ G $-representation
$ \mathop{\rm ind} \hspace{-0,1cm} \mid^G_H  \hspace{-0,1cm} (\BB) $,
 as a category of $ H $-equivariant objects in a certain category,
was constructed and studied in~\cite{GK}. (See also\cite{G}, where the study of $ G $-equivariant objects was continued.)
}
\end{nt}

\begin{nt} {\em
Analogous categorical induced representations, depending on \linebreak $ (a_1, \ldots, a_n) \in \dc^n $, can be defined for the group
$ GL_{2n} (\da_{\Delta}) $, where $ \da_{\Delta} $ is some subring in the adelic ring  $ \da_X $ or $ \da_{X}^ {\rm ar} $ of two-dimensional normal integral scheme $ X $ of finite type over $ \dz $ (see \S~\ref {C2}). In this case, one has to use the central extension
 $ \widehat{GL_2 (\da_{\Delta})}_{\dr_+^*} $ of the group $ GL_{2} (\da_{\Delta}) $ by the group $ \dr_+^* $ (this central extension
  was constructed in \S~\ref{C2}).
Instead of the groups $ \CCC_m $
one has to take the following subgroups $ \CCC_{\{m_C \}} $ of the group $ \prod_{x \in C} GL_{2n} (K_{x, C}) $:
$$ \CCC_{\{m_C \}}
= GL_{2n} (\da_{\Delta}) \cap \left(1 + \prod_{C \subset X} t_C^{m_C} \prod_{x \in C} M_{2n} (\oo_{K_{x, C}}) \right) \mbox{,} $$
where $ C $ runs over a set of all  one-dimensional integral subschemes of $ X $, $ t_C $ is a local parameter in the field $ K_C $,
and the set $ \{m_C \} $ consists of the integers $ m_C \ge 0 $ such  that $ m_C = 0 $ for almost all $ C $ (here, not all the integers $ m_C $ are zero.) If $ m_C = 0 $ for any integral one-dimensional subscheme $ C $, then
$$ \CCC_{\{0 \}}
= GL_{2n} (\da_{\Delta}) \cap \prod_{x \in C} GL_{2n} (\oo_{K_{x, C}}) \mbox{.}
$$
 If $ \da_{\Delta} $ coincides with a two-dimensional local field $ K $, then the resulting categorical representation of the group
$ GL_{2n} (K) $ is exactly the categorical representation $ \VV_{a_1, \ldots, a_n} $ which was constructed above.}
\end{nt}

\begin{nt} {\em
By virtue of remark~\ref{zam}, for any subgroup $ H \subset GL_{2n} (K) $ there is the action of the group
$ H $ on a $ \dc $-vector space
$ \Hom_{\VV_{a_1, \ldots, a_n}} (E, D) $ for any objects $ (E, \theta) $, $ (D, \vartheta) $ of the category of
$ (\VV_{a_1 , \ldots, a_n})^H $. It would be interesting to compare these actions with  the actions of  groups from~\cite{KazhG}, where the representation theory of reductive groups over two-dimensional local fields was constructed in the category
$ {\rm Pro} \Vect_{\dc} $, where $ \Vect_{\dc} $ is the category of all vector spaces over the field $ \dc $.}
\end{nt}

\subsection{Some hypothesis} \label{hyp}
By analogy with section~\ref{one-dim} we will formulate some hypothesis
about  smooth spherical  actions of the group $ GL_{2n}(K) $ on
generalized $ 2 $-vector spaces over the field $ \dc $, where $ K $
is a two-dimensional local field.

Let $ \BB $ be an abelian $ k $-linear category with a $ G $-action,
where $ G $ is a group and $ k $ is a field.

We will {\em say} that a pair $ (\ad, \mathfrak{T}) $  is a
$G$-subrepresentation of the category $ \BB $ if $ \ad $ is an
abelian $ k $-linear category with a $ G $-action , and $
\mathfrak{T} = (T, \ve) $ is a $ G $-linear functor from $ \ad $ to
$ \BB $ such that  the functor $ T $ is fully faithful and exact.

We will {\em  say} that the action of the group $ G $ on the
category $ \BB $ is irreducible if for any $ G $-subrepresentation $
(\ad, \mathfrak{T}) $ of the category $ \BB $ we have that either
the category $ \ad $ is equivalent to a category consisting of only
zero objects or the functor $ T $ is essentially surjective (in the
latter case the functor $ T $ is an equivalence of the categories $
\ad $ and $ \BB $, i.e the categories $ \ad $ and $ \BB $ are
$G$-equivalent, see remark~\ref{G-equiv}).

We will {\em say } that a pair $ (\CC, {\mathfrak P}) $ is a $
G$-quotient representation of the category $ \BB $ if $ \CC $ is an
abelian $ k $-linear category with a $ G $-action, $ \mathfrak{P}
=(P, \ve) $ is a $ G $-linear functor from $ \BB $ to $ \CC $. In
addition,  we demand here that $ P = Q \circ R $, where $ R $ is the
functor from $ \BB $ to the category $ \BB / \ad $ for some Serre
subcategory $ \ad $ of  $ \BB $, and $ Q $ is an equivalence of
categories  $ \BB / \ad $ and $ \CC $. (We note that the functor $
P$ is an exact functor.)

The definition of a spherical (or unramified) action has to be
stronger than the demanding of the
 irreducibility of an action of
the group $ GL_{2n} (K) $ (where $ n \ge 1 $ and $ K $ is a
two-dimensional local field) on the category $ \BB $ with the
condition that the abelian category $ \BB^{GL_{2n} (\oo_K)} $ is not
equivalent to a category consisting of only zero objects.  It can be
explained as follows. The subring $ \oo_K \hookrightarrow K $ does
not depend on the discrete valuation on the residue field, while
unramified extensions of the field  $ K $ which were considered are
unramified also with respect to the discrete valuation on the
residue field $ \bar{K} $. Therefore, one has  also to take into
account the discrete valuation on the residue field $ \bar{K} $.

The correct definition of a spherical (unramified) action of the
group $ GL_{2n} (K) $ would allow to state the following hypothesis.

{\em Any smooth spherical action of the group $ GL_{2n} (K) $ on a $
\dc $-linear abelian category $ \BB $  can be obtained as a
subquotient representation (i.e. as a quotient representation of a
subrepresentation) of the  category $ \VV_{a_1, \ldots, a_n } $ for
some  \linebreak  $ (a_1, \ldots, a_n) \in \dc^n / \left(\frac{2 \pi i}{\ln q}
\right) \dz^n $. (We omitted  here an indication to $ GL_{2n} (K)
$-linear functors which set the quotient representation and the
subrepresentation.) }

\vspace{0.5cm}

We next consider the case when $ n = 1 $ and a category  $ \BB $ is
equivalent to the category $ \Vect_{\dc}^{\rm fin} $, i.e  it  is a
one-dimensional $ 2 $-vector space.  In this case, it is easy to
understand that any action of the group $ GL_ {2} (K) $ on the
category $ \BB $ is irreducible. According to section~\ref{odn}, an
action of the group $ GL_2 (K) $ on the category $ \BB $ up to an
equivalence corresponds to a central extension (this is a one-to-one
correspondence):
\begin{equation} \label{ccc} 1 \lto \dc^* \lto
\hat{G} \lto GL_2 (K) \lto 1 \mbox{.}
\end{equation}
Clearly,  the condition of the sphericity of the action of the group
$ GL_2 (K) $ on the category $ \BB $ must match with a condition that
central extension~\eqref{ccc} is obtained from  central
extension~\eqref{cenext}:
$$ 1 \lto K_2 (K) \lto \widehat{GL_2
(K)} \lto GL_2 (K) \lto 1 $$ by means of  a map
\begin{equation} \label{oto}
K_2 (K) \lto \dc^* \: \ (f, g) \mapsto q^{-a \, \nu_K (f, g)}
\mbox{,}
\end{equation}
where $ q $ is the number of elements in the last residue field $ K
$ and $ a \in \dc / \left(\frac{2 \pi i}{\ln q} \right) \dz $.
When
this condition is satisfied, the action of the group $ GL_2 (K) $ on the category
$ \BB$ is smooth, because the central extension $ \hat {G} $  splits
over the subgroup  $ GL_2 (\oo_K) $. (Indeed, when  this condition is satisfied, it follows
from proposition~\ref{pr}  that the central extension
$\hat {G} $ is obtained from the central extension
$ \widehat{GL_2 (K)}_{\dr_+^*} $ by means of the  map $ \dr_+^* \lto \dc^*: x \mapsto x^ a $. By construction,
the central extension
$ \widehat{GL_2 (K)}_{\dr_+^*} $ splits over the subgroup
 $GL_2 (\oo_K) $.) In addition, according to  statement~\ref{st3} of theorem~\ref{th2}, the category $ \BB $ is a
 $ GL_2 (K)$-subrepresentation of the category $ \VV_a $ (here we have omitted an indication
to a $ GL_{2} ( K) $-linear functor which sets the subrepresentation).

Now we consider a two-dimensional local field $ K $ which is isomorphic to one of the following fields (see~\cite[\S2]{O1}, where  it is described,  how these fields arise from  algebraic surfaces and from arithmetic  surfaces):
\begin{equation} \label{2f}
\df_q ((u)) ((t)) \ \mbox{,} \qquad \qquad L ((t))) \ \mbox{,} \qquad \qquad L \{\{u \} \} \mbox{,}
\end{equation}
where $ L $ is a one-dimensional local field of characteristic zero, and the field $ L \{\{u \} \} $ is the completion of the field
$ \Frac (\oo_L [[u]]) $ by  a discrete valuation given by the  height $ 1 $ prime  ideal $ m_L \oo_L [[u]] $  of the ring $ \oo_L [[u]] $
(the ideal $ m_L $ is  the maximal  ideal of the ring $ \oo_L $).
We consider a subring  $ B $ of the field $ K $ which is given  in accordance with  isomorphism~\eqref{2f} by one of the following ways:
$$
\df_q [[u]] ((t)) \ \mbox{,} \qquad \qquad \oo_L ((t))) \ \mbox{,} \qquad \qquad L \cdot \oo_L [[u]] \mbox{.}
$$
We note that the subring $ B $ depends on the choice of isomorphism~\eqref{2f}. Therefore, we fix some such subring  $ B $   for a  field $ K $.

Let $ E $ be a simple object of the category  $ \BB $. We suppose that
$$ (E, \theta) \in \Ob (\BB^{GL_2 (\oo_K)}) \mbox{,} \qquad (E, \vartheta) \in
\Ob (\BB^{GL_2 (B)}) \mbox{,} \qquad (E, \upsilon) \in \Ob (\BB^{K^*}) $$
for some $ \theta $, $ \vartheta $ and $ \upsilon $, where $ K^* \hookrightarrow GL_2 (K) $
(an embedding  to the upper left corner). Then it is not difficult to see that $ \theta $, $ \vartheta $ and $ \upsilon $ define splittings of
 the group $ \hat{G} $ over the subgroups
$ GL_2 (\oo_K) $, $ GL_2 (B) $ and $ K^* $, respectively. This assumption will gives the condition of the sphericity for the action of the group $ GL_2 (K) $ on the category $ \BB $, where $ \BB $ is a one-dimensional $2$-vector space, since  the following proposition holds.
\begin{prop} \label{pred}
Let a two-dimensional local field $ K $ be isomorphic to one of the fields of the form~\eqref{2f}. A central extension  $ \hat{G} $ is obtained from the central extension  $ \widehat{GL_2 (K)} $ by means of the map~\eqref{oto} (for some element $ a \in \dc /
\left(\frac{2 \pi i}{\ln q} \right) \dz $) if and only if
 the central extension  $ \hat{G} $ splits over the subgroups $ GL_2 (\oo_K) $, $ GL_2 (B) $ and $ K^* $ of the group  $ GL_2 (K) $ (the subgroup $ K^* $ is embedded into the upper left corner).
\end{prop}
\begin {nt}
{\em
Although the subring $ B $ depends on the choice of isomorphism~\eqref {2f}, the obtained result on the description of the central
extension $ \hat{G} $ by
means of map~\eqref{oto} does not depend on the choice of this isomorphism.
}
\end{nt}
\begin{proof} (of proposition~\ref{pred}).
 We assume, first, that the central extension $ \hat {G} $ is obtained from the central extension $ \widehat{GL_2 (K)} $ by means of
 the map $ K_2 (K) \lto \dc^*: (f, g) \mapsto q^{- a \, \nu_K (f, g)} $. Since, by construction, the central extension
 $ \widehat{GL_2 (K)} $ splits over the subgroup  $ K^*$ of the group $ GL_2 (K) $,
we have that the central extension $ \hat {G} $ will also  split over this subgroup.
 It follows from proposition~\ref{pr}  that the central extension  $ \hat{G} $ is obtained from the central extension
 $ \widehat{GL_2 (K)}_{\dr_+^*} $
by means of the map $ \dr_+^* \lto \dc^*: x \mapsto x^a $. (For the construction of the central extension
$ \widehat{GL_2 (K)}_{\dr_+^*} $
 one has to obtain the field $ K $ from a scheme $ \Spec R $, where a ring $ R $ is either $ \df_q [[u, t]] $ or $ \oo_L [[t]] $ or
 $ \oo_L [[ u]] $.) Therefore it is sufficient to prove that the central extension  $ \widehat{GL_2 (K)}_{\dr_+^*} $ splits over the subgroups
 $ GL_2 (\oo_K) $ and $ GL_2 (B) $ of the group  $ GL_2 (K) $.
For the subgroup $ GL_2 (\oo_K) $ this follows immediately from the construction of the central extension.

We prove the splitting of the central extension $ \widehat{GL_2 (K)}_{\dr_+^*} $ over the  subgroup $ GL_2 (B) $ of the group $ GL_2 (K) $.
First,  we  show that the central extension $ \widetilde{GL_2 (K)}_{\dr_+^*} $ (see section~\ref{constr}) splits over the subgroup
 $ GL_2 (B) $. For every $ g \in GL_2 (K) $ we can uniquely define
an element  $ \mu_{B, g} \in \mu (\oo_K^{\, 2} \mid g \oo_K^{\, 2}) $ as $ \mu_{B, g} = \mu_1^{- 1} \otimes \mu_2 $, where
 $ \mu_1 \in \mu (\oo_K^{\, 2} / h \oo_K^{\, 2}) $,
$ \mu_2 \in \mu (g \oo_K^{\, 2} / h \oo_K^{\, 2}) $, the element $ h \in GL_2 (K) $ is any  that satisfies
$ h \oo_K^{\, 2} \subset \oo_K^{\, 2} $,
$ h \oo_K^{\, 2} \subset g \oo_K^{\, 2} $. The elements  $ \mu_1 $ and $ \mu_2 $ are defined  by the following rules:
 $$ \mu_1 \left(\frac{B^2 \cap \oo_K^{\, 2}}{B^2 \cap h \oo_K^{\, 2}} \right) = 1 \mbox{,}
\qquad \mu_2 \left(\frac{B^2 \cap g \oo_K^{\, 2}}{B^2 \cap h \oo_K^{\, 2}} \right) = 1 $$
(this definition makes sense, since the spaces in the brackets at $ \mu_1 $ and $ \mu_2 $ are  open compact subgroups of
the locally compact groups $ \oo_K^{\, 2} / h \oo_K^{\, 2} $ and $ g \oo_K^{\, 2} / h \oo_K^{\, 2} $, respectively.) The resulting element
$ \mu_ {B, g} $ does not depend on the choice of the element $ h \in GL_2 ( K) $.
Now, from the construction it is not difficult to see that a section $ g \mapsto (g, \mu_{B, g}) $, where $ g \in GL_2 (B) $, gives the splitting of the central extension $ \widetilde{GL_2 (K) }_{\dr_+^*} $ over the subgroup  $ GL_2 (B) $. (We have to use that $ g B^2 = B^2 $ for any
$ g \in GL_2 (B) $.)
The  group $ GL_2 (B) = SL_2 (B) \rtimes B^* $. Hence  it follows that the central extension
$ \widehat{GL_2 (K)}_{\dr_+^*} $ splits over the subgroup  $ GL_2 (B) $. (The action of the group $ B^* $ does not change the newly constructed section over the subgroup $ SL_2 (B) $. This is obvious if one  fixes also this constructed section over the  subgroup  $ B^* $.)

We suppose now that the central extension $ \hat{G} $ splits over the subgroups $ GL_2 (\oo_K) $, $ GL_2 (B) $ and $ K^* $ of the group
 $ GL_2 (K) $. It follows from  proposition~\ref{prop-coh} and the splitting of the central extension $ \hat {G} $ over the subgroup $ K^* $  that the central extension $ \hat {G} $ is
obtained from the central extension $ \widehat{GL_2 (K)} $ by means of some map $ \phi \in \Hom (K_2 (K), \dc^*) $.
It follows
from the description of the map $ \nu_K (\cdot, \cdot): K_2 (K) \lto \dz $ as the composition of the boundary maps in the Milnor $ K $-theory  that the group
$ \Ker \nu_K (\cdot, \cdot) $ is generated by the following elements: $ (f, g) $, where $ f $ and $ g $ are from  $ \oo_K^* $, and $ (f, g) $, where $ f $ and $ g $ are from $ B^* $.
(Indeed, $ \nu_K (\cdot, \cdot) = \partial_1 \circ \partial_2 $ and the following sequences are exact:
\begin{eqnarray*}
1 \lto K_2 (\oo_K) \lto K_2 (K) \stackrel{\partial_2}{\lto} \bar{K}^* \lto 1
\mbox{,} \\
1 \lto \oo_{\bar{K}^*} \lto \bar{K}^* \stackrel{\partial_1}{\lto} \dz \lto 0 \mbox{,}
\end{eqnarray*}
 where $ \partial_2 $ is the tame symbol  associated with the discrete valuation of the field $ K $, $ \partial_1 = \nu_{\bar{K}} $.)

It follows from formula~\eqref{equ}  that if the central extension $ \hat {G} $ splits over the subgroup  $ GL_2 (\oo_K) $, then
  the elements  $ (f, g) $, where $ f $ and $ g $ are from $ \oo_K^* $,
belong to the group $ \Ker \phi $. Similarly, it follows from the splitting of the central extension $ \hat {G} $ over the subgroup $ GL_2 (B) $
  that the elements  $ (f, g) $, where $ f $ and $ g $
are from $ B^* $,
belong to the group $ \Ker \phi $.
Hence, we obtain that $ \phi = \varphi \circ \nu_K (\cdot, \cdot) $ for some homomorphism $ \varphi $ from the group $ \dz $ to the
group $ \dc^* $. We find
$ a \in \dc / \left(\frac{2 \pi i}{\ln q} \right) \dz $ such  that
$ q^{-a} = \varphi (1) $.
Then $ \phi = q^{-a \, \nu_K (\cdot, \cdot)} $. The proposition is proved.
\end{proof}

\vspace{0.3cm}

\noindent
Steklov Mathematical Institute of RAS \\
Gubkina str. 8, 119991, Moscow, Russia \\
{\it E-mail:}  ${d}_{-} osipov@mi.ras.ru$ \\

\end{document}